\documentclass{siamart190516}

\usepackage[T1]{fontenc}
\usepackage{mystyle}
\usepackage{pgf,tikz,pgfplots}
\usepackage{tikz-3dplot}
\usetikzlibrary{arrows,arrows.meta}
\usepackage{subfig}
\usepackage{enumerate}
\usepackage{bbold}
\usepackage{hyperref}

\newcommand{\nint}[1]{[#1]}

\DeclareMathOperator{\ext}{ext}
\DeclareMathOperator{\conv}{conv}
\DeclareMathOperator{\aff}{aff}

\newcommand{\Sol}{X}

\newcommand{\x}{x}
\newcommand{\R}{\mathbb R}
\newcommand{\A}{\mathcal A}

\DeclareMathOperator{\dir}{dir}
\DeclareMathOperator{\ri}{ri}

\newsiamremark{remark}{Remark}

\title{The Geometry of Sparse Analysis Regularization\thanks{\textbf{Fundings:}
 This work was partly supported by ANR GraVa ANR-18-CE40-0005 and Projet ANER RAGA G048CVCRB-2018ZZ.}}
\author{
  Xavier Dupuis\thanks{Université de Bourgogne, Dijon, France (\email{xavier.dupuis@u-bourgogne.fr}).}
  \and
  Samuel Vaiter\thanks{CNRS \& Université C\^ote d'Azur, Nice, France (\email{samuel.vaiter@math.cnrs.fr}).}}
\date{}
\begin{document}

\maketitle

\begin{abstract}
  Analysis sparsity is a common prior in inverse problem or machine learning including special cases such as Total Variation regularization, Edge Lasso and Fused Lasso.
We study the geometry of the solution set (a polyhedron) of the analysis $\ell^1$ regularization (with $\ell^2$ data fidelity term) when it is not reduced to a singleton without any assumption of the analysis dictionary nor the degradation operator.
In contrast with most theoretical work, we do not focus on giving uniqueness and/or stability results, but rather describe a worst-case scenario where the solution set can be big in terms of dimension.
Leveraging a fine analysis of the sub-level set of the regularizer itself, we draw a connection between support of a solution and the minimal face containing it, and in particular prove that extreme points can be recovered thanks to an algebraic test.
Moreover, we draw a connection between the sign pattern of a solution and the ambient dimension of the smallest face containing it.
Finally, we show that any arbitrary sub-polyhedra of the level set can be seen as a solution set of sparse analysis regularization with explicit parameters.
\end{abstract}

\section{Introduction}
\label{sec:intro}

We focus on a convex regularization promoting sparsity in an analysis dictionary in the context of a linear inverse problem/regression problem where the regularization reads:
\begin{equation}\label{eq:regularization}
  \min_{\x \in \RR^n} \frac{1}{2} \norm{y - \Phi \x}_2^2 + \lambda \normu{D^* x} 
\end{equation}
where $y \in \RR^q$ is an observation/response vector, $\Phi \colon \RR^n \rightarrow \RR^q$ is the sensing/acquisition linear operator, $D \colon \RR^p \rightarrow \RR^n$ is a dictionary and $\lambda>0$ the hyper-parameter used as a trade-off between fidelity and regularization.
Note that at this point, we do not make any assumption on the dictionary $D$ or the acquisition operator $\Phi$.

This convex regularization is known as analysis $\lun$-regularization~\cite{Elad_2007} in the inverse problems community or generalized Lasso~\cite{tibshirani2011} in statistics.
Let us mention that it includes several popular regularizers as special cases such that (anisotropic) total variation~\cite{Rudin1992Nonlineartotalvariation} when $D$ is a discrete difference operator, wavelet coefficient analysis~\cite{steidl2004equivalence} using a wavelet transform as an analysis dictionary or fused Lasso~\cite{tibshirani2005sparsity} when using the concatenation of the identity matrix and a discrete difference operator, i.e., using a Lasso regularization with an additional constraint on the (discrete) gradient.
In the noiseless context, when $y \in \Im \Phi$, the following constrained formulation is used instead of the Tikhonov formulation~\eqref{eq:regularization} as
\begin{equation}\label{eq:regularizationNONOISE}
  \min_{\x \in \RR^n} \normu{D^* x} \qsubjq \Phi x = y .
\end{equation}
We focus here on the noisy version of the regularization in order to keep our discussion concise.
The purpose of this paper is to answer the following question:
\begin{quote}
  \emph{When the solution set of~\eqref{eq:regularization} is not reduced to a singleton, what is its ``geometry''?}
\end{quote}

One possible motivation could be to study some \emph{generalized solution path} of such a problem, not with respect to the hyper-parameter $\lambda$ (see e.g. \cite{10.1214/009053604000000067,mairal2012complexity, tibshirani2011}) but with respect to some parameter of $D$.
For example, consider 
\[
D_\rho = \begin{pmatrix}
  \rho & 0 \\ 0 & 1
\end{pmatrix}, \quad
\Phi = \begin{pmatrix}
 1 & 1
\end{pmatrix}, \quad
y=2, \quad \lambda =1 ;
\]
one can show that the solution set is
\[
\left\{ (1,0) \right\} \text{ if } \rho < 1, \quad
\left[ (1,0) , (0,1) \right] \text{ if } \rho = 1, \quad
\left\{ (0,1) \right\} \text{ if } \rho > 1.
\] 
Even if most of the time the solution set is reduced to a singleton (see below),
it is essential to describe what happens when it is not the case to understand the behavior
of the generalized solution path.
In this paper, we do not tackle fully this multivalued point of view, and leave the sensitivity analysis for future work.

\subsection{Previous works}

\paragraph{Uniqueness certificate of analysis regularization}
Among several theoretical issues, sufficient condition for uniqueness of the solution set of~\eqref{eq:regularization} have been extensively studied,
see for instance~\cite{6380620,nam2013cosparse,2018arXiv180507682A,zhang2016one,tardivel:hal-03262087}.
Several uniqueness conditions can be proposed, where the simplest is for instance requiring $n \leq q$ and $\Phi$ having full rank: the $\ell^2$-loss term is strictly convex, and uniqueness follows from it.
The task of studying the case when the solution set is not reduced to a singleton can be seen as rather formal since most of the time the solution set is reduced to a singleton~\cite{tibshirani2011,6380620}, but nevertheless, it exhibits interesting properties of sparse analysis regularization.
See \cref{sec:uniqueness} for a discussion of some of these conditions.

\paragraph{Solution set of generic convex program}
Describing the geometry of the solution set in convex optimization has been a subject of intense study starting from the work of~\cite{mangasarian1988simple} and its generalization to non-smooth convex program~\cite{burke1991characterization}. Several extensions have been proposed such as~\cite{jeyakumar1995characterizing} for pseudo-linear programs or in a different setting (minimization of concave function), \cite{mangasarian1999minimum} shows that one can describe one solution with minimal sparsity level.

\paragraph{Representer theorems}
We shall also remark that coming from the statistics community, \cite{KIMELDORF197182} (and popularized in~\cite{scholkopf2001generalized}) initiates a line of work coined as representer theorems, culminating recently in~\cite{boyer2018representer} and \cite{unser2019unifying}.
The basic idea of these kinds of results is to show that under some assumption, one can write every element of the solution set of a convex program as a sum of elementary atoms.

\paragraph{Description of polytopes}
Convex polytopes and polyhedrons are central objects in geometry~\cite{ziegler1995lectures} and convex analysis.
A part of our results provides a connection between faces and signs of vector living in the analysis domain.
We can draw a connection with the study of oriented matroid~\cite{bjorner1999oriented} and zonotopes~\cite{bolker1969class} (analysis $\ell^1$-ball are zonotopes) as described in~\cite[Lecture 7]{ziegler1995lectures}, in particular in section 7.3.

\subsection{Contributions}
In contrast to these lines of work, we take here a more direct and specific approach.
We give below an overview of our contributions.

\paragraph{Geometry of the analysis $\ell^1$-ball}
The first part of our work (\eqref{sec:unit}) is dedicated to studying the geometry of the analysis $\ell^1$-ball.
We study across several results the direction and the relative interior of the intersection between the sub-level set of the regularizer and another set. We refine our analysis progressively starting from any convex component of the level set, then looking to sub-polyhedra of the sub-level set ending by the faces itself of the level set.
We show several specific results:
\begin{itemize}
  \item The sign pattern defines a bijection between the set of exposed faces of the analysis $\ell^1$-ball and the set of feasible signs in the dictionary $D$ as proved in~\eqref{prop:inclusion-sign}.
  \item The extreme points of the analysis $\ell^1$-ball can be recovered with a purely algebraic result thanks to \eqref{cor:NSCextremality}. We draw a link between our result and a remark in~\cite{boyer2018representer} which is a topological argument
\end{itemize}

\paragraph{Geometry of the solution set}
Thanks to the study of the analysis $\ell^1$-ball, we give in a second part (\eqref{sec:struct}) consequences on the solution set of~\eqref{eq:regularization}.
We show that:
\begin{itemize}
  \item Using \eqref{lem:affine_compo} and \eqref{prop:itself}, we describe the geometry of the solution set in \eqref{prop:cons-sol-faces}. 
  \item The solution set of~\eqref{eq:regularization} admits extreme points if and only if, the condition denoted by ($H_0$) and assumed to hold all throughout~\cite{6380620} or~\cite{vaiter2013local}, namely $\Ker \Phi \cap \Ker D^* = \{0\}$, holds.
  In this case, the extreme points are precisely those which satisfy the condition denoted by ($H_J$) in \cite{vaiter2013local} at a given solution to perform a sensitivity analysis, see \eqref{prop:conn-litt}.
  \item For any affine space which intersect non-trivially the unit-sphere, one can find $\Phi$, $y$ such that the solution set is exactly this intersection, see \eqref{thm:arb} and \eqref{prop:arbitrary_sol}.
\end{itemize}

\subsection{Notations}

For a given integer $n$, the set of all integers between $1$ and $n$ is denoted by $\nint{n} = \ens{1,\dots,n}$.

\paragraph{Vectors and support}
Given $u \in \RR^m$, the support $\supp(u)$ and the sign vector $\sign(u)$ are defined by
\begin{equation*}
  \supp(u) = \enscond{i \in \nint{m}}{u_i \neq 0} \qandq \sign(u) = (\sign(u_i))_{i \in \nint{m}},
\end{equation*}
and its cardinal is coined the $\ell^0$-norm $\norm{u}_0 = \abs{\supp(u)}$.
The cosupport $\cosupp(u)$ is the set $\cosupp(u) = \nint{m} \setminus \supp(u)$.
Given $u,v \in \RR^m$, the inner product is written $\dotp{u}{v} = \sum_{i=1}^m u_i v_i$ and the associated norm is written $\norm{u}_2 = \sqrt{\dotp{u}{u}}$.
We will also use the $\ell^1$-norm $\norm{u}_1 = \sum_{i=1}^m \abs{u_i}$ and $\ell^\infty$-norm $\norm{u}_\infty = \max_{i\in\nint{m}} \abs{u_i}$.

\paragraph{Linear operators}
Given a linear operator $D \in \RR^{n \times m}$, $D^* \in \RR^{m \times n}$ is the transpose operator, $D^+ \in \RR^{n \times m}$ its Moore--Penrose pseudo-inverse, $\Ker D \subseteq \RR^m$ its null-space and $\Im D \in \RR^n$ its column-space.
Given $I \subseteq \nint{m}$, $D_I \in \RR^{n \times |I|}$ is the matrix formed by the column of $D$ indexed by $I$.
The identity operator is denoted $\Id_m$ or $\Id$.
Given a vector $u \in \RR^m$, $x_I$ is the vector of components indexed by $I$.
Given a subspace $F \subseteq \RR^m$, we denote by $\Pi_F$ the orthogonal projection on $F$.
Given a vector $u \in \RR^n$, its diagonalized matrix $\diag(u) \in \RR^{n \times n}$ is the diagonal matrix such that $\diag(u)_{ii} = u_i$ for every $i \in \nint{n}$.

\paragraph{Convex analysis}
Given a convex, lower semicontinuous, proper function $f : \RR^m \to \RR$, its sub-differential $\partial f$ is given
\begin{equation*}
  \partial f(u) = \enscond{\eta \in \RR^m}{f(u) \geq f(v) + \dotp{\eta}{u-v}} .
\end{equation*}
Given a convex set $C$, the affine hull $\aff(C)$ is the smallest affine set containing $C$, the direction $\dir(C)$ of $C$ is the direction of $\aff(C)$ and its relative interior $\ri(C)$ is the interior of $C$ relative to its affine hull $\aff(C)$.
The relative boundary $\rbd(C)$ of $C$ is the boundary of $C$ relative to $\aff(C)$.
The dimension $\dim(C)$ of $C$ is the dimension of $\aff(C)$.
We say that $x \in C$ is an extreme point if there are no two different $x_1,x_2 \in C$ such that $x = \frac{x_1+x_2}{2}$.
The set of all extreme points of $C$ is denoted by $\ext(C)$.
For instance, given two points $x_1 \neq x_2 \in \RR^n$, the segment $C = [x_1,x_2]$ is such that its affine hull is $\aff(C) = \enscond{x_1 + t x_2}{t \in \RR}$, its relative interior is the open segment $\ri(C) = (x_1,x_2)$, its relative boundary and set of extreme points $\ext(C) = \rbd(C) = \ens{x_1,x_2}$, its dimension is 1 and its direction is $\dir(C) = \RR (x_1 - x_2)$.

\subsection{Examples of operators $D^{*}$}

We illustrate our results in this paper on different analysis regularization settings. In particular, we focus our interest on different operators:
\begin{itemize}
  \item The Lasso~\cite{tibshirani1996regression}, corresponding to $D = D_{\text{Lasso}} = \Id$, used to recover sparse vectors.
  \item The Total Variation regularization~\cite{Rudin1992Nonlineartotalvariation}, and more specifically the 1D Total Variation, i.e., when $D : \RR^n \to \RR^{n-1}$ is a forward difference operator on $n$ points:
        \begin{equation*}
          D^* =
          \begin{pmatrix}
            -1 & +1 & 0 & \cdots & 0 \\
            0 & -1 & +1 & \ddots & \vdots \\
            \vdots &  \ddots & \ddots & \ddots & 0 \\
            0 & \cdots & 0 & -1 & +1 \\
          \end{pmatrix} .
        \end{equation*}
        This is a popular prior in image processing to regularize ``cartoon'' or piecewise regular images.
  \item More generally, we consider the Graph-Total Variation regularization~\cite{sharpnack2012sparsitency} where $D = D_{G}$ is the vertex-edge incidence matrix of a graph $G$.
        Specific instance include the 1D and anisotropic 2D Total Variation~\cite{Rudin1992Nonlineartotalvariation}, Cluster Lasso~\cite{she2010sparse}.
\end{itemize}


\section{The unit ball of the sparse analysis regularizer}
\label{sec:unit}

This section contains the core of our results.
After giving preliminary results on sign vectors in \cref{sec:sign-prelem}, we show that the unit ball is a convex polyhedron by giving its half-space representation in \cref{sec:unit-ball}.
Then, we study properties of convex subset of the unit-sphere in \cref{sec:convex-sphere} which lead us to \cref{lem:convex-subset} which turns to be the foundation of latter results.
\Cref{sec:subpolyhedra} contains a sequence of results which represent our main contribution: \cref{lem:affine_compo} which describes in detail the affine components of the unit-ball, \cref{prop:itself} which instantiates this result to setting of an affine component included in the unit sphere, \cref{prop:faces_of_inter} which extends this result to any exposed faces and finally \cref{prop:NSCextremality} which gives a necessary and sufficient condition of extremality.
Finally, in \cref{sec:cons-unit}, we reformulate our previous results in order to describe the exposed faces of the unit-ball, and to show that there exists a bijection between the set of exposed faces and feasible signs. We also draw a connection to the work of~\cite{boyer2018representer}.

\subsection{Preliminary results on sign vectors} \label{sec:sign-prelem}

We first define an order on the set of all possible signs $\{-1, 0,+1\}^p$ along with a notion of consistency of signs which can be related to the idea of ``sub-signs''.
\begin{definition} \label{def:sign_order}
Let $s,s' \in \{-1, 0,+1\}^p$. We say that 
\begin{itemize}
 \item $s \preceq s'$ if for all $i \in [p]$, $s_i \ne 0  \ \Rightarrow s'_i = s_i$;
 \item $s$ and $s'$ are consistent if for all $i \in [p]$, $s_i \ne 0$ and $s_i' \ne 0\ \Rightarrow s'_i = s_i$.
\end{itemize}
\end{definition}

The following remarks connect the notion of support/cosupport to this sign pattern.
\begin{remark} \label{rmk:sign_order}
\begin{enumerate}
 \item $s \preceq s' \Rightarrow \supp(s) \subset \supp(s') \Leftrightarrow  \cosupp(s') \subset \cosupp(s)$;
 \item If $s$ and $s'$ are consistent, then 
 \[s \preceq s' \Leftrightarrow   \supp(s) \subset \supp(s')  \Leftrightarrow  \cosupp(s') \subset \cosupp(s);\]
 \item If $s \preceq s''$ and $s' \preceq s''$, then $s$ and $s'$ are consistent;
 \item The set $\ens{-1, 0,+1}^p$ endowed with the order relation $\preceq$ is a poset.
\end{enumerate}
\end{remark}

The following lemma gives a characterization of the $\ell^1$-norm which will be used intensively in latter results.
\begin{lemma} \label{lem:sign&norm}
 Let $\theta \in \R^p$. Then for any $s \in \ens{-1, 0,+1}^p$, $ \langle s,\theta \rangle \le \| \theta \|_1$,
 and the equality holds if and only if $\sign(\theta) \preceq s$.
\end{lemma}

\begin{proof}
We prove the result component-wisely.
Let $\alpha \in \R$. Then for any $s \in \ens{-1, 0,+1}$, $s \alpha \le |\alpha|$. 
Suppose now that $\sign(\alpha) \preceq s$. If $\alpha = 0$, then $s\alpha = |\alpha|$, and if $\alpha \ne 0$,
then $s = \sign(\alpha)$ and $s \alpha = |\alpha|$.
Conversely, suppose that $\sign(\alpha) \npreceq s$. Then $\alpha \ne 0$ and $s \ne \sign(\alpha)$, i.e. $s = 0$ or $s = -\sign(\alpha)$.
In both cases, $s \alpha < |\alpha|$.
\end{proof}

\subsection{Half-space representation of the unit ball} \label{sec:unit-ball}

We denote by $B_1$ (resp. $\partial B_1$) the unit ball (resp. the unit sphere), or sub-level set (resp. level set) for the value $1$,
of the sparse analysis regularizer $R \colon x \mapsto \| D^* x \|_1$:
\begin{align*}
 B_1 & = \{ x \in \R^n : \| D^*x \|_1 \le 1 \} ,\\
 \partial B_1 & = \{ x \in \R^n : \| D^*x \|_1 = 1 \}.
\end{align*}
Since $R$ is one-homogeneous, the results of this section apply to all sub-level sets for positive values.

\begin{proposition}\label{prop:unit-ball}
The unit ball $B_1$ is a full-dimensional convex polyhedron, a half-space representation of which is given by
  \begin{equation*}
    B_1 = 
    \bigcap_{s \in \{-1,0,1\}^p} \{ x \in \R^n : \langle D s , x \rangle \le 1 \} .
  \end{equation*}
\end{proposition}

\begin{proof}
  First note that $B_1$ has a nonempty interior (in particular $0 \in B_1$), namely $\{ x \in \R^n : \| D^*x \|_1 < 1 \}$, which is equivalent for a convex set to be of full dimension.
  
  Second denote $A = \bigcap_{s \in \ens{-1,0,1}^p} \enscond{x}{\dotp{D s}{x} \leq 1}$.
  Let $x \in B_1$. By \cref{lem:sign&norm}, $\langle D s, x \rangle = \langle s, D^* x \rangle \le \| D^* x \|_1 \le 1$
  for any $s \in \ens{-1,0,1}^p$ so $x \in A$.
  Conversely, let $x \in A$ and $s = \sign(D^*x)$. By \cref{lem:sign&norm}, $\| D^* x \|_1 = \langle s, D^* x \rangle = \langle D s, x \rangle \le 1$
  so $x \in B_1$. Then $B_1 = A$ is a convex polyhedron.
\end{proof}

Note that this half-space representation is redundant, and if $D = \Id$, then it is the $\ell^{1}$-ball.
The general question of the minimal representation of $H$-polyhedron is known to be hard, we shall leave it to future work.
However, we can use this proposition to derive a way to construct exposed face of $B_1$ as claimed in the following lemma.
\begin{lemma} \label{lem:supporting-hyperplane}
 Let $\bar s \in \{-1,0,1\}^p$. Then
 \[
  B_1 \cap \{ x \in \R^n : \langle D\bar s , x \rangle = 1 \} =
  \partial B_1 \cap \{ x \in \R^n : \sign(D^*x) \preceq \bar s \} ;
 \]
  it is either empty or an exposed face of $B_1$.
\end{lemma}

\begin{proof}
 Let $x \in B_1$. Then by \cref{lem:sign&norm}, $\langle D\bar s , x \rangle = 1$
 if and only if $1 = \langle \bar s ,D^* x \rangle  \le \| D^* x \|_1 \le 1$,
 if and only if $\| D^* x \|_1 = 1$ and $\sign(D^*x) \preceq \bar s$.
 If the intersection is nonempty, then $\{ x \in \R^n : \langle D\bar s , x \rangle = 1 \}$ is a supporting hyperplane of $B_1$
 by \cref{prop:unit-ball}, and thus its intersection with $B_1$ is an exposed face of the polyhedron.
\end{proof}

For a given $\bar s$, the set $\{ x \in \R^n : \sign(D^*x) \preceq \bar s \}$ looks hard to describe.
In fact, there exists a linear representation as told in the following lemma.
\begin{lemma}\label{lem:linear_sign}
	Let $\bar s \in \{-1,0,1\}^p$. Then
	\[
		\{ x \in \R^n : \sign(D^*x) \preceq \bar s \}
		=
		\{ x \in \RR^n : D_{\bar J}^* x = 0 \text{ and } \diag(\bar s_{\bar I}) D_{\bar I}^* x \geq 0 \}
	\]
	where $\bar J = \cosupp(\bar s)$ and $\bar I = \supp(\bar s)$.
\end{lemma}
\begin{proof}
	It is a straightforward rewriting of $\sign(D^*x) \preceq \bar s$.
	Indeed,
	\begin{align*}
		\sign(D^*x) \preceq \bar s 
		& \Leftrightarrow D_{\bar J}^* x = 0 \text{ and } \bar s_{i} (D^* x)_i \geq 0, \forall i \in \bar{I} \\
		& \Leftrightarrow D_{\bar J}^* x = 0 \text{ and } \diag(\bar s_{\bar I}) D_{\bar I}^* x \geq 0 .
	\end{align*}
\end{proof}
Note that we can exchange the role of $\bar s$ and $D^* x$, and we also obtain that
\[
	\{ x \in \R^n : \sign(D^*x) \preceq \bar s \}
	=
	\{ x \in \RR^n : D_{\bar J}^* x = 0 \text{ and } \diag(D_{\bar I}^* x) \bar s_I  \geq 0 \}.
\]

\subsection{Convex components of the unit sphere} \label{sec:convex-sphere}

In this section we consider nonempty convex subsets $C\subset \partial B_1$.
All the results will hold in particular for exposed faces of $B_1$.

We begin with a lemma on general convex sets.

\begin{lemma} \label{lem:general-convex-subset}
 Let $X$ be a nonempty convex set and $C \subset X$ be a nonempty convex subset. 
 Suppose that there exists an exposed face $G$ of $X$ such that $\ri(C) \cap G \ne \emptyset$.
 Then $C \subset G$.
 Moreover, if $G$ is exposed in $\aff(X)$, then $G \subset \rbd(X)$.
 \end{lemma}

 \begin{proof}
  Recall that an exposed face of $X$ is defined as $G =X \cap \{ x : \langle  \alpha , x \rangle = \beta \}$
  with $\{ x : \langle  \alpha , x \rangle = \beta \}$ a supporting hyperplane of $X$, i.e. such that
  $X \subset \{ x : \langle  \alpha , x \rangle \le \beta \}$.
  Suppose that $C \not \subset G$ and let $x \in C \setminus G \subset X$ and $\bar x \in \ri(C) \cap G$ (nonempty).
  Then $\langle \alpha , x \rangle < \beta$ and $\langle \alpha , \bar x \rangle = \beta$.
  Let $d = \bar x -x \in \dir(C)$. Since $\bar x \in \ri(C)$,
 $\bar x + \varepsilon d \in C \subset X$ for $|\varepsilon|$ small.
 But $\langle \alpha , \bar x + \varepsilon d \rangle = \beta + \varepsilon (\beta - \langle \alpha , x \rangle ) > \beta$
 for $\varepsilon > 0$, which is a contradiction. Then $C \subset G$.
 
 It is a classical result that $G \subset \bd(X)$, see e.g. \cite[Part III, Section 2.4]{hiriart2013convex}.
 If $G$ is exposed in $\aff(X)$, we get that $G \subset \rbd(X)$ by considering $\aff(X)$ as the ambient space. 
 \end{proof}

The following lemma is the first result of a long number of consequences which study the direction and relative interior of the intersection of the unit ball with another set.
\begin{lemma} \label{lem:convex-subset}
 Let $\bar x \in \partial B_1$, $\bar s = \sign(D^* \bar x)$, $\bar J = \cosupp(D^*\bar x)$, and
 \[
 \bar F  = B_1 \cap \{ x \in \R^n : \langle D\bar s , x \rangle = 1 \} .
 \]
 \begin{enumerate}[(i)]
  \item $\bar F = \partial B_1 \cap \{ x \in \R^n : \sign(D^*x) \preceq \bar s \}$ and it is an exposed face;
  \item $C \subset \bar F$ for any nonempty convex subset $C \subset B_1$ such that $\bar x \in \ri(C)$;
  \item $ \dir(\bar F) =   (D \bar s)^\bot \cap  \Ker D^*_{\bar J}$;
  \item $\ri(C) \subset \ri(\bar F)$ for any nonempty convex subset $C \subset B_1$ such that $\bar x \in \ri(C)$;
  \item $\ri(\bar F) =\partial B_1 \cap \{ x \in \R^n : \sign(D^*x) = \bar s \}$.
 \end{enumerate}
\end{lemma}

\begin{proof}
 (i) The expression for $\bar F$ is given by \cref{lem:supporting-hyperplane}.
 It follows that $\bar x \in \bar F$ and thus $\bar F$ is an exposed face of $B_1$.

 (ii) Let $C$ be a nonempty convex subset of $B_1$ such that $\bar x \in \ri(C)$.
 Then $\ri(C) \cap \bar F \ne \emptyset$, and by \cref{lem:general-convex-subset}, $C \subset \bar F$.

(iii) The inclusion $\dir(\bar F) \subset (D \bar s)^\bot$ follows from the definition of $\bar F$. 
Moreover, for any $x \in \bar F$, $\sign(D^*x)\preceq \bar s$, which implies that $\bar J \subset \cosupp(D^*x)$, i.e. $D^*_{\bar J}x = 0$.
Then $ \bar F\subset \Ker D^*_{\bar J}$, and $\dir(\bar F) \subset \Ker D^*_{\bar J}$.
Conversely, let $d \in (D \bar s)^\bot \cap  \Ker D^*_{\bar J}$.
Since $d \in \Ker D^*_{\bar J}$ and $\sign(D^*\bar x) = \bar s$, $\cosupp(D^*\bar x) \subset \cosupp(D^*d)$ and
$\sign(D^*(\bar x + \varepsilon d) ) \preceq \bar s$ for $|\varepsilon|$ small.
Then by \cref{lem:sign&norm}, $\| \bar x + \varepsilon d \|_1 = \langle \bar s, D^*(\bar x + \varepsilon d ) \rangle
= \langle D \bar s, \bar x + \varepsilon d \rangle = \langle D \bar s, \bar x  \rangle = 1$
since $d \in (D \bar s)^\bot$. Then $\bar x + \varepsilon d \in \bar F$ for $|\varepsilon|$ small, and $d \in \dir(\bar F)$.

(iv)-(v) First, we prove that $\partial B_1 \cap \{ x \in \R^n : \sign(D^*x) = \bar s \} \subset \ri(\bar F)$
(thus in particular $\bar x \in \ri(\bar F)$).
Let $x \in \partial B_1$ be such that $\sign(D^* x) = \bar s$.
By~(iii) and its proof, for any $d \in \dir(\bar F)$, $x + \varepsilon d \in \bar F$ for $|\varepsilon|$ small. Thus $x \in \ri(\bar F)$.
Second, let $C$ be a nonempty convex subset of $B_1$ such that $\bar x \in \ri(C)$.
By~(ii), $C \subset \bar F \subset \partial B_1$.
Let $\hat x \in \ri(C)$ and $\hat s = \sign(D^* \hat x)$. 
Note that $\hat s \preceq \bar s$ since $\hat x \in \bar F$.
Let $\hat F = B_1 \cap \{ x \in \R^n : \langle D\hat s , x \rangle = 1 \}$.
Applying~(ii) to $\hat x$ and $C = \bar F$, we get that $\bar F \subset \hat F$,
which implies that $\bar s \preceq \hat s$. Thus $\sign(D^* \hat x) = \bar s$
and $\ri(C) \subset \partial B_1 \cap \{ x \in \R^n : \sign(D^*x) = \bar s \}$.
This proves (iv) as well as the missing inclusion of (v) by setting $C = \bar F$ (recall that $\bar x \in \ri(\bar F)$).

\end{proof}

The following proposition is a direct consequence of \cref{lem:convex-subset} which allows to characterize faces of $B_1$ by an arbitrary convex subset of it.
\begin{proposition} \label{prop:unique-face_C}
Let $C$ be a nonempty convex subset of $\partial B_1$.
Let $\bar x \in \ri(C)$, $\bar s= \sign(D^*\bar x)$, and
 \[
  \bar F = B_1 \cap \{ x \in \R^n : \langle D\bar s , x \rangle = 1 \} .
 \]
Then $C \subset \bar F$ and $\ri(C) \subset \ri(\bar F)$.
 Moreover, $\bar F$ is the smallest face of $B_1$ such that $\ri(C) \cap \bar F \ne \emptyset$
 and the unique face of $B_1$ such that $\ri(C) \cap \ri(\bar F) \ne \emptyset$.
\end{proposition}

\begin{proof}
By \cref{lem:convex-subset}~(i) and (iii), $C \subset \bar F$ and $\ri(C) \subset \ri(\bar F)$.
Let $F$ be a face such that $\ri(C)\cap F \ne \emptyset$. 
If $F = B_1$, then $\bar F \subset F$; otherwise $F$ is an exposed face since $B_1$ is a polyhedron,
and by \cref{lem:general-convex-subset}, $C\subset F$. 
It follows that $ \emptyset \ne \ri(C)\cap \ri(\bar F) \subset F \cap \ri(\bar F)$, and by \cref{lem:general-convex-subset} again,
$\bar F \subset F$.
Suppose now that $\ri(C) \cap \ri(F) \ne \emptyset$.
Then permuting $F$ and $\bar F$, we get that $F \subset \bar F$, thus $F = \bar F$.
\end{proof}

\begin{remark}
The uniqueness actually holds with the same proof for a general nonempty convex set $X$: given a nonempty convex subset 
$C \subset \rbd(X)$, there exists at most one exposed face $G$ of $X$ such that $\ri(C) \cap \ri(G) \ne \emptyset$.
The existence reduces to the existence, for any $x \in \rbd(X)$, of an exposed face $G$ of $X$ such that $x \in \ri(G)$.
\end{remark}

For a singleton $C = \{ \bar x \}$, the previous proposition becomes the following.

\begin{corollary} \label{cor:singleton}
	Let $\bar x \in \partial B_1$ and $\bar s = \sign(D^* \bar x)$. Then 
$ \bar F = B_1 \cap \{ x \in \R^n : \langle D\bar s , x \rangle = 1 \} $
	is the smallest face of $B_1$ such that $\bar x \in \bar F$ 
	and the unique face of $B_1$ such that $\bar x \in \ri(\bar F)$.
\end{corollary}

We can also derive from
\cref{prop:unique-face_C} and \cref{lem:convex-subset}
the following properties about the mapping $x \mapsto \sign(D^*x)$.

\begin{corollary} \label{cor:sign_map}
Let $C$ be a nonempty convex subset of the unit-sphere $\partial B_1$.
Then $\max_{x \in C}\sign(D^*x)$ is well-defined
and this maximum is attained everywhere in $\ri(C)$. In particular $\sign(D^*\cdot)$ is constant on $\ri(C)$.
\end{corollary}

For any $x,x'\in C$, $[x,x']$ is a nonempty convex subset of $\partial B_1$
and thus $\sign(D^*\cdot)$ is constant on $]x,x'[$.
Moreover, since $s = \sign(D^*x)$ and $s'= \sign(D^*x')$ are both $\preceq \max_{x \in C}\sign(D^*x)$,
they are consistent (see \cref{rmk:sign_order}). It follows that the constant value $s''$ of $\sign(D^*\cdot)$ on $]x,x'[$
can be given explicitly:
\[
s''_i = \begin{cases}
	s_i & \text{if } s_i \ne 0, \\
	s'_i & \text{if } s'_i \ne 0, \\
	0 & \text{otherwise}.
	\end{cases} 
\]
It is also the maximum of $\sign(D^*\cdot)$ over $[x,x']$.

Finally we get a general sufficient condition of extremality.

\begin{corollary} \label{cor:SCextremality}
	Let $C$ be a nonempty convex subset of $\partial B_1$. 
	Let $\bar x \in C$ and $\bar s= \sign(D^*\bar x)$.
	If $\bar x$ is the unique $x \in C$ such that $\sign(D^*x)\preceq \bar s$,
	then $\bar x \in \ext(C)$.
\end{corollary}

\begin{proof}
	Let $x_1,x_2 \in C$ such $\bar x = \frac{x_1 +x_2}{2}$.
	Then $[x_1, x_2]$ is a nonempty convex subset of $\partial B_1$.
	By \cref{cor:sign_map}, $\sign(D^*x) = \bar s$ for any $ x \in ]x_1,x_2[$.
	By uniqueness of $\bar x$, $]x_1,x_2[ = \{\bar x \}$, i.e. $x_1 = x_2 = \bar x$ and thus $\bar x$ is an extreme point.
\end{proof}
Observe that the uniqueness condition in the \cref{cor:SCextremality} can be written as $C \cap \bar F = \{\bar x \}$
with $\bar F = B_1 \cap \{ x \in \R^n : \langle D\bar s , x \rangle = 1 \} $
the smallest face of $B_1$ containing $\bar x$.

\subsection{Sub-polyhedra of the unit ball} \label{sec:subpolyhedra}

In this section we consider nonempty convex polyhedra of the form $\A \cap B_1$
with $\A$ an affine subspace.
The results on convex components of the unit sphere apply to such sets if $\A \cap B_1 \subset \partial B_1$,
and in any case, as we will see, to exposed faces of such polyhedra.
Again, the results of this section will hold in particular for exposed faces of $B_1$.

We begin with a useful lemma.

\begin{lemma} \label{lem:ri-inter}
Let $\A$ be an affine subspace and $C$ be a nonempty convex set such that $\A \cap \ri(C) \ne \emptyset$.
Then
\[
 \ri(\mathcal A \cap C) = \A \cap \ri(C) \text{ and } \dir(\mathcal A \cap C) = \dir(\A) \cap \dir(C).
\]
\end{lemma}

\begin{proof}
 Since $\ri(\A) = \A$, we have $\ri(\A) \cap \ri(C) \ne \emptyset$.
 Then by \cite[Part III, Proposition~2.1.10]{hiriart2013convex}, $\ri(\mathcal A \cap C) = \ri(\A) \cap \ri(C) = \A \cap \ri(C)$.
 Let us now prove that $ \dir(\mathcal A \cap C) = \dir(\mathcal A) \cap \dir(C)$.
 Let $d \in  \dir(\mathcal A \cap C)$ and $\bar x \in \ri(\mathcal A \cap C)$ (nonempty).
 Then $\bar x + \varepsilon d \in \mathcal A \cap C$ for $|\varepsilon|$ small, and $d \in \dir(\mathcal A) \cap \dir(C)$.
 Similarly, let $d \in \dir(\mathcal A) \cap \dir(F)$ and $\bar x \in \ri(\mathcal A) \cap \ri(C)$ (nonempty).
 Then $\bar x + \varepsilon d \in \mathcal A \cap C$ for $|\varepsilon|$ small, and $d \in \dir(\mathcal A \cap C)$.
\end{proof}

The following lemma will be used in \cref{lem:affine_compo}.
\begin{lemma} \label{lem:dirA}
			Let $\A$ be an affine subspace such that $\emptyset \ne \A \cap B_1 \subset \partial B_1$.
			Let $\bar x \in \A \cap \partial B_1$, $\bar s= \sign(D^* \bar x)$, and $\bar J = \cosupp(D^*\bar x)$. Then
			\begin{enumerate}[(i)]
				\item $\dir(\A ) \subset \bigcup_{s \succeq \bar s}  \left\{ d \in \R^n : \langle Ds, d \rangle \ge 0 \right\}$;
				\item $\dir(\A)^\bot \cap \left( \sum_{s \succeq \bar s} \R_+ Ds \right) \ne \{0\}$;
				\item $\dir(\A) \cap \Ker D^*_{\bar J} \subset (D \bar s)^\bot$,
				 i.e. $ \dir(\mathcal A) \cap (D \bar s)^\bot \cap \Ker D^*_{\bar J} =\dir(\mathcal A) \cap \Ker D^*_{\bar J}$.
				\item $\mathcal A  \cap \{ x \in \R^n : \sign(D^*x) \preceq \bar s \} \subset \partial B_1$,
				i.e. $\mathcal A \cap \partial B_1 \cap \{ x \in \R^n : \sign(D^*x) \preceq \bar s \}
				= \mathcal A \cap \{ x \in \R^n : \sign(D^*x) \preceq \bar s \}$;
				\item $\mathcal A  \cap \{ x \in \R^n : \sign(D^*x) = \bar s \} \subset \partial B_1$,
				i.e. $\mathcal A \cap \partial B_1 \cap \{ x \in \R^n : \sign(D^*x) = \bar s \}
				= \mathcal A \cap \{ x \in \R^n : \sign(D^*x) = \bar s \}$.
			\end{enumerate}
		\end{lemma}
	
	\begin{proof}
		(i) Suppose that the inclusion does not hold and let $d \in \dir(\A)$ such that $\langle Ds, d \rangle < 0$ for all $s \succeq \bar s$.
		Then $\langle Ds, \bar x + \varepsilon d \rangle < 1$ for all $s$ and $\varepsilon > 0$ small. Indeed, by \cref{lem:sign&norm},
		if $s \succeq \bar s$, then $\langle Ds, \bar x + \varepsilon d \rangle < 1$ for $\varepsilon > 0$;
		otherwise, $\langle Ds, \bar x + \varepsilon d \rangle < 1$ for $\varepsilon$ small.
		Still by \cref{lem:sign&norm}, $\| \bar x + \varepsilon d  \|_1 <1$ for $\varepsilon > 0$ small,
		i.e. $\bar x + \varepsilon d \in \A \cap \mathring B_1$,
		which is in contradiction with $\A \cap B_1 \subset \partial B_1$.
		
		(ii) Suppose that the intersection is reduced to $\{0\}$ and consider the dual cone of both sides of the expression
		(we denote by $C^*=\{ y \in \R^n :  \forall x \in C, \langle y,x\rangle \ge 0\}$ of a subset $C \subset \R^n$).
		Since $\dir(\A)^\bot$ and $\left( \sum_{s \succeq \bar s} \R_+ Ds \right) \ne \{0\}$ are two polyhedral cones,
		we get that
		\[
		\dir(\A)^{\bot*}  + \Big( \sum_{s \succeq \bar s} \R_+ Ds \Big)^* = \R^n,
		\]
		where $	\dir(\A)^{\bot*} = \dir(\A)$ and 
		$\Big( \sum_{s \succeq \bar s} \R_+ Ds \Big)^* = \bigcap_{s \succeq \bar s} Ds ^*$.
		It follows from (i) that $\R^n \subset \bigcup_{s \succeq \bar s}  Ds ^*$, which is not true
		(note e.g. that $d = x-\bar x$ with $x \in \mathring B_1$ is such that $\langle Ds, d \rangle < 0)$ for all $s \succeq \bar s$).
		
		(iii) First note that for $d \in \Ker D^*_{\bar J}$, $\langle Ds, d \rangle = \langle D\bar s, d \rangle$ for all $s \succeq \bar s$.
		Indeed, $(D^* d)_i = 0$ if $\bar s_i = 0$ and $s_i = \bar s_i$ if $\bar s_i \ne 0$, thus $s_i (D^* d)_i = \bar s_i(D^* d)_i$ for all $i$
		and $\langle s, D^*d \rangle = \langle \bar s, D^*d \rangle$.
		Let now $d \in \dir(\A) \cap \Ker D^*_{\bar J}$. Since by (i) there exists $s \succeq \bar s$ such that $\langle Ds, d \rangle \ge 0$,
		it follows that $\langle D\bar s, d \rangle \ge 0$. And since $-d \in \dir(\A) \cap \Ker D^*_{\bar J}$ too, we get that
		$\langle D\bar s, d \rangle = 0$, which proves the inclusion and the equivalent equality.
		
		(iv) Let $x \in \A$ such that $\sign(D^*x) \preceq \bar s$ (in particular, $\bar J \subset \cosupp(D^*x)$)
		and let $d = x - \bar x$. Then $d \in \dir(\A) \cap \Ker D^*_{\bar J}$ (recall that $\bar x \in \A$ and $\bar J = \cosupp(D^*\bar x)$),
		and by (ii), $d\in (D \bar s)^\bot$.
		By \cref{lem:sign&norm}, $\| D^*x\|_1 = \langle \bar s , D^*x \rangle = \langle D\bar s , \bar x + d \rangle
		= \langle D\bar s , \bar x \rangle = \| D^*\bar x\|_1 = 1$, which proves the inclusion and the equivalent equality.
		
		(v) The proof is the same as for (iv).
	\end{proof}

The following theorem is similar to \cref{lem:convex-subset} when we replace convex subset by sub-polyhedra (here of the unit sphere).
\begin{theorem} \label{lem:affine_compo}
	Let $\A$ be an affine subspace intersecting $\partial B_1$.
	Let $\bar x \in \A \cap \partial B_1$, $\bar s= \sign(D^* \bar x)$, $\bar J = \cosupp(D^*\bar x)$, and
	$\bar F = B_1 \cap \{ x \in \R^n : \langle D\bar s , x \rangle = 1 \}$.
	Then 
	\[
	\bar G = \A \cap \bar F
	\]
	\begin{enumerate}[(i)]
		\item is the smallest face of $\A \cap B_1$ such that $\bar x \in \bar G$
		and the unique face of $\A \cap B_1$ such that $\bar x \in \ri(\bar G)$
		($\bar G$ is possibly equal to $\A \cap B_1$ itself);
		\item satisfies the following:
		\begin{align*}
		\bar G &=  \mathcal A \cap \partial B_1 \cap \{ x \in \R^n : \sign(D^*x) \preceq \bar s \}, \\
		\ri(\bar G) &  = \mathcal A \cap \partial B_1 \cap \{ x \in \R^n : \sign(D^*x) = \bar s \},\\ 
		\dir(\bar G) &  = \dir(\mathcal A) \cap (D \bar s)^\bot \cap \Ker D^*_{\bar J};
		\end{align*}
		\item satisfies the following, in the case where $\A \cap B_1 \subset \partial B_1$:
		\begin{align*}
		\bar G &=  \mathcal A \cap \{ x \in \R^n : \sign(D^*x) \preceq \bar s \},  \\
		\ri(\bar G) &  = \mathcal A  \cap \{ x \in \R^n : \sign(D^*x) = \bar s \}, \\
		\dir(\bar G) &  = \dir(\mathcal A) \cap \Ker D^*_{\bar J}. 
		\end{align*}
		
	\end{enumerate}	
\end{theorem}

\begin{proof} (i) First note that $\bar G = (\A \cap B_1) \cap \{ x : \langle D\bar s , x \rangle = 1 \}$
	is an exposed face of $\A \cap B_1$ (in particular it is a convex subset of $\A \cap B_1$) and is such that
	$\bar x \in \ri(\bar G) = \A \cap \ri(\bar F)$ by \cref{lem:ri-inter}.
	Let $G$ be a face of $\A \cap B_1$ such that $\bar x \in G$.
	If $G = \A \cap B_1$, then $\bar G \subset G$; otherwise $G$ is an exposed face since $\A \cap B_1$ is a convex polyhedron,
	and by \cref{lem:general-convex-subset}, $\bar G \subset G$.
	Suppose now that $\bar x \in \ri(G)$. Then permuting $G$ and $\bar G$, we get that $G \subset \bar G$, thus $G = \bar G$.
	
	(ii) By definition, $\bar G = \A \cap \bar F$ and $\bar x \in \A \cap \ri(\bar F)$.
	By \cref{lem:ri-inter}, $\ri(\bar G) = \A \cap \ri(\bar F)$ and $\dir(\bar G) = \dir(\A) \cap \dir(\bar F)$.
	The expression of these sets follows from \cref{lem:convex-subset}.
	
	(iii) We proved the strengthened expression of the previous sets in \cref{lem:dirA}.
\end{proof}

We get the next result on $\A \cap B_1$ itself in the case where it is a subset of $\partial B_1$.

\begin{proposition} \label{prop:itself}
	Let $\A$ be an affine subspace such that $\emptyset \ne \A \cap B_1 \subset \partial B_1$.
	Then \[ \A \cap B_1 = \A \cap \bar F \]
	with $\bar F = B_1 \cap \{ x \in \R^n : \langle D\bar s , x \rangle = 1 \} $ and $\bar s = \max_{x \in \A \cap B_1} \sign(D^* x)$
	(or equivalently $\bar s = \sign(D^*\bar x)$ for some $\bar x \in \ri(\A \cap B_1)$).
	In particular, the results of \cref{lem:affine_compo}~(iii) hold for $\A \cap B_1$.
\end{proposition}

\begin{proof}
	Since $\A \cap B_1$ is a nonempty convex subset of $\partial B_1$, $\bar s$ is well-defined by \cref{cor:sign_map}.
	Let $\bar x \in \ri(\A \cap B_1) \subset \A \cap \partial B_1$. Then by \cref{lem:affine_compo}~(i),
	$\A \cap \bar F$ is the unique face of $\A \cap B_1$ containing $\bar x$ in its relative interior;
	it is thus equal to the face $\A \cap B_1$.
\end{proof}

In the general case, we can describe all the exposed faces of $\A \cap B_1$. 

\begin{proposition} \label{prop:faces_of_inter}
	Let $\A$ be an affine subspace such that $\emptyset \ne \A \cap B_1 \ne \aff(\A \cap B_1)$.
	Let $G$ be a face $\A \cap B_1$ exposed in $\aff(\A \cap B_1)$.
	Then \[G = \A \cap \bar F \]
	with $\bar F = B_1 \cap \{ x \in \R^n : \langle D\bar s , x \rangle = 1 \} $ and $\bar s = \max_{x \in G} \sign(D^* x)$
	(or equivalently $\bar s = \sign(D^*\bar x)$ for some $\bar x \in \ri(G)$).
	In particular, the results of \cref{lem:affine_compo}~(ii) 
	(or (iii) if $\A \cap B_1 \subset \partial B_1$) hold for $G$.
\end{proposition}

\begin{proof}
	By \cref{lem:general-convex-subset}, $G \subset \rbd(\A \cap B_1)$. 
	Let us show that $\rbd (\A \cap B_1) \subset \A \cap \partial B_1$.
	We distinguish two cases: if $\A \cap B_1 \subset \partial B_1$, there is nothing to prove;
	otherwise, $\mathcal A \cap \mathring B_1 \ne \emptyset$. By \cref{lem:ri-inter} applied the full-dimensional convex $B_1$,
	$ \ri(\mathcal A \cap B_1)  = \ri(\mathcal A) \cap \ri(B_1) = \mathcal A \cap \mathring B_1 $.
	Then $\rbd(\mathcal A \cap B_1) = \mathcal A \cap B_1 \setminus \mathcal A \cap \mathring B_1 = \mathcal A \cap \partial B_1$.
	In particular, $G$ is a nonempty convex subset of $\partial B_1$, thus $\bar s$ is well-defined by \cref{cor:sign_map}.
	
	Let $\bar x \in \ri(G) \subset \A \cap \partial B_1$. Then by \cref{lem:affine_compo}~(i),
	$\A \cap \bar F$ is the unique face of $\A \cap B_1$ containing $\bar x$ in its relative interior;
	it is thus equal to the face $G$.
\end{proof}

In this setting, we get a necessary and sufficient condition of extremality.
Note that the notion of extremality can be related to the topology of the set, here our condition only use an algebraic characterization.
\begin{proposition} \label{prop:NSCextremality}
Let $\A$ be an affine subspace intersecting $\partial B_1$.
Let $\bar x \in \A \cap \partial B_1$, $\bar s= \sign(D^*\bar x)$, and $\bar J = \cosupp(D^*\bar x)$.
Then 
\begin{align*}
	\bar x \in \ext(\A \cap B_1) & \Leftrightarrow \dir(\mathcal A) \cap (D \bar s)^\bot \cap \Ker D^*_{\bar J} = \{0\} \\
								 & \Leftrightarrow \dir(\mathcal A) \cap \Ker D^*_{\bar J} = \{0\} \text{ in the case where } \A \cap B_1 \subset \partial B_1.
\end{align*}
Moreover, $\A \cap B_1$ admits extreme points (that necessarily belong to $\A\cap \partial B_1$) 
if and only if it is compact (i.e. is a convex polytope),
if and only if $\dir(\mathcal A) \cap \Ker D^* = \{0\}$.
\end{proposition}

\begin{proof}
	Since $\A \cap B_1$ is a convex polyhedron, $\bar x \in \ext(\A \cap B_1)$ $\Leftrightarrow$
	$\{ \bar x \} $ is a face $\A \cap B_1$ $\Leftrightarrow$  $\{ \bar x \}= \bar G$ of \cref{lem:affine_compo}
	$\Leftrightarrow$ $\dir(\bar G) = \{0\}$.
	
	Recall that extreme points belong to $\rbd(\A \cap B_1)$, and that as in \cref{prop:faces_of_inter},
	$\rbd(\A \cap B_1) \subset \A \cap \partial B_1$.
	Note also that we always have $\dir(\mathcal A) \cap \Ker D^* \subset \dir(\mathcal A) \cap (D \bar s)^\bot \cap \Ker D^*_{\bar J}$.
	Thus if $\A \cap B_1$ admits extreme points, then $\dir(\mathcal A) \cap \Ker D^* = \{0\}$ by the beginning of the corollary,
	which implies that $\A \cap B_1$ is bounded and thus compact, which in turn implies the existence of extreme points
	\cite[Part III, Proposition~2.3.3]{hiriart2013convex}.
\end{proof}

\begin{remark} \label{rmk:CSextremality}
	It is possible to show directly (via \cref{cor:SCextremality})
	that the condition above is sufficient.
	Indeed, assume that $\dir(\mathcal A) \cap (D \bar s)^\bot \cap \Ker D^*_{\bar J} = \{0\}$.
	Let $x \in \A \cap \partial B_1$ such that $\sign(D^*x) \preceq \bar s$;
	let us show that $x = \bar x$: $x - \bar x \in \dir(\A)$; by \cref{lem:sign&norm}, $\langle \bar s, D^* x \rangle  = \| D^* x \|_1 
	= \| D^* \bar x \|_1  = \langle \bar s, D^* \bar x \rangle$, thus $x - \bar x \in (D \bar s)^\bot$;
	by \cref{rmk:sign_order}, $\bar J \subset \cosupp(D^*x)$, thus $x - \bar x \in \Ker D^*_{\bar J}$.
 	Then $\bar s$ is minimal, and by \cref{cor:SCextremality}, $\bar x \in \ext(\A \cap B_1)$.
\end{remark}

\subsection{Consequence results on the unit ball} \label{sec:cons-unit}
The previous results will be at the core of our study of the solution set of~\eqref{eq:regularization} in \cref{sec:struct}.
Nevertheless, we can also dive deeper into this analysis in order to fully characterize the faces of the unit-ball as a byproduct.

A first consequence or reformulation of the previous results is that all the exposed faces of $B_1$ are of the form
$ B_1 \cap \{ x \in \R^n : \langle D\bar s , x \rangle = 1 \} $.

\begin{proposition} \label{prop:face2sign}
	Let $F$ be an exposed face of $B_1$.
	Then 
	\[
	F = B_1 \cap \{ x \in \R^n : \langle D\bar s , x \rangle = 1 \}
	\]
	with $\bar s = \max_{x \in F} \sign(D^* x)$
	(or equivalently $\bar s = \sign(D^*\bar x)$ for some $\bar x \in \ri(F)$). 
	Moreover,
	\begin{align*}
	F &  = \partial B_1 \cap \{ x \in \R^n : \sign(D^*x) \preceq \bar s \} \\
	& = \{ x \in \R^n : \langle D\bar s , x \rangle = 1 \} \cap \{ x \in \R^n : \sign(D^*x) \preceq \bar s \} , \\
	\ri(F) &  = \partial B_1 \cap \{ x \in \R^n : \sign(D^*x) = \bar s \} \\
	&  = \{ x \in \R^n : \langle D\bar s , x \rangle = 1 \} \cap \{ x \in \R^n : \sign(D^*x) = \bar s \}, \\
	\dir(F) &  = (D \bar s)^\bot \cap \Ker D^*_{\bar J} \text{ (where } \bar J = \cosupp(\bar s)\text{)}.
	\end{align*}
\end{proposition}

\begin{proof}
	Since $F$ is a nonempty convex subset of $\partial B_1$, $\bar s$ is well-defined by \cref{cor:sign_map}.
	The first statement is a consequence of \cref{cor:singleton}.
	The next statements follow from \cref{lem:convex-subset} and \cref{prop:itself} with $\A = \{ x \in \R^n : \langle D\bar s , x \rangle = 1 \}$
	a supporting hyperplane of $B_1$ defining the face $F$ (recall or note by \cref{lem:general-convex-subset} that $\A \cap B_1 \subset \partial B_1$
	for any $\A$ supporting hyperplane of $B_1$).
\end{proof}

\begin{remark}
	Any exposed face $F$ of $B_1$ satisfies $\Ker D^* \subset \dir(F) \subset (D \bar s )^\bot$
	with equality if and only if $(D \bar s)^\bot \cap \Ker D^*_{\bar J} = \Ker D^*$ (for the left inclusion)
	and $(D \bar s)^\bot \subset  \Ker D^*_{\bar J}$ (for the right inclusion).
\end{remark}

Together with \cref{lem:linear_sign}, the previous proposition gives
the following half-space representations of the exposed faces of $B_1$.
\begin{corollary}
	Let $F$ be an exposed face of $B_1$.
	Then 
	\[
	F = \{ x \in \R^n : \langle D\bar s , x \rangle = 1, D^*_{\bar J} x = 0,  \diag(\bar s_{\bar I})D^*_{\bar I} x \ge 0 \}
	\]
	with $\bar s = \max_{x \in F} \sign(D^* x)$, $\bar J = \cosupp(\bar s)$, and $\bar I = \supp(\bar s)$.
\end{corollary}

By \cref{lem:convex-subset} and \cref{prop:face2sign},
the mapping $F \mapsto \max_{x \in F} \sign(D^*x)$ is a bijection between
the set of exposed faces of $B_1$ and the set of feasible signs $\{ \sign(D^*x) : x \in \partial B_1\}$.
The next observation is that this bijection preserves the partial orders (i.e. it is an order isomorphism).

\begin{proposition} \label{prop:inclusion-sign}
 Let $F_1$ and $F_2$ be two exposed faces of $B_1$ and 
 \[
	s_i = \max_{x \in F_i} \sign(D^*x), \ i=1,2. 
 \]
Then
 \[
 F_1 \subset F_2 \Leftrightarrow  s_1 \preceq s_2 .
 \]
 In this case, denoting by $J_1 = \cosupp(s_1)$,
 	\begin{align*}
 F_1 & = F_2 \cap \{ x \in \R^n : J_1 \subset \cosupp(D^*x)  \} , \\
 \ri(F_1) &  = F_2 \cap \{ x \in \R^n : J_1 = \cosupp(D^*x)  \}, \\
 \dir(F_1) &  = \dir(F_2) \cap \Ker D^*_{J_1} .
 \end{align*}
\end{proposition}

\begin{proof}
 Suppose that $F_1 \subset F_2$ and let $x \in \ri(F_1) \subset F_2$; then $s_1 = \sign(D^*x) \preceq s_2$. 
 Conversely, suppose that $s_1 \preceq s_2$ and let $x \in F_1$; then $\sign(D^*x) \preceq s_1 \preceq s_2$, thus $x\in F_2$.
 
 Assume now that these two conditions hold. Then, we have $F_1 \subset F_2 \cap  \{ x : J_1 \subset \cosupp(D^*x)  \}$ since
 $F_1 \subset \{ x : J_1 \subset \cosupp(D^*x)  \}$
 (recall \cref{rmk:sign_order}: $\sign(D^*x) \preceq s_1$ implies that $J_1 \subset \cosupp(D^*x)$). 
 Conversely, let $x \in F_2 \cap \{ x : J_1 \subset \cosupp(D^*x)  \}$.
 Since $\sign(D^* x)$ and $s_1$ are both $\preceq s_2$, they are consistent and thus 
 the converse is true: $J_1 \subset \cosupp(D^*x)$ implies that $\sign(D^* x) \preceq s_1$. 
 And since $x \in F_2 \subset\partial B_1$, $x \in F_1$.
 The same proof holds for $\ri(F_1)$ and $\{ x : J_1 = \cosupp(D^*x)  \}$.
 For $\dir(F_1)$, note that $\Ker D^*_{J_1} \subset \Ker D^*_{J_2}$
 (since $J_2 = \cosupp(s_2) \subset J_1$). Then $\dir(F_2) \cap \Ker D^*_{J_1} = (Ds_2)^\bot \cap \Ker D^*_{J_1}$
 (and $\dir(F_1) = (Ds_1)^\bot \cap \Ker D^*_{J_1}$).
 By the proof of \cref{lem:dirA}~(ii), for $d \in \Ker D^*_{J_1}$,
 $\langle Ds, d \rangle = \langle Ds_1, d \rangle$ for all $s \succeq s_1$. In particular,
 $(Ds_1)^\bot \cap \Ker D^*_{J_1} = (Ds_2)^\bot \cap \Ker D^*_{J_1}$,
 which concludes the proof.
 
\end{proof}

 In the spirit of \cite{boyer2018representer},
 we consider the compact polyhedron $(\Ker D^*)^\bot \cap B_1$,
 which is isomorphic to the projection of $B_1$ onto the quotient
 of the ambient space by the lineality space $\Ker D^*$.
 \cref{prop:NSCextremality} gives the following necessary and sufficient condition of extremality.
 
\begin{corollary} \label{cor:NSCextremality}
	The convex polyhedron $(\Ker D^*)^\bot \cap B_1$ is compact (i.e. is a convex polytope).
	It admits extreme points that belong to $(\Ker D^*)^\bot \cap \partial B_1$.
	Given $\bar x \in (\Ker D^*)^\bot \cap \partial B_1$, $\bar s= \sign(D^*\bar x)$, and $\bar J = \cosupp(D^*\bar x)$,
	\[
	\bar x \in \ext((\Ker D^*)^\bot \cap B_1)  \Leftrightarrow (\Ker D^*)^\bot \cap (D \bar s)^\bot \cap \Ker D^*_{\bar J} = \{0\}.
	\]
\end{corollary}

\begin{remark}
	In \cite[Section~4.1.3]{boyer2018representer}, the authors notice that 
	since $(\Ker D^*)^\bot \cap B_1$ and $\Im D^* \cap B_{\ell^1}$ are in bijection through $D^*$
	and its pseudo-inverse $(D^*)^+$, it holds that 
	\[
	\bar x \in \ext((\Ker D^*)^\bot \cap B_1)  \Leftrightarrow \bar x = (D^*)^+\bar \theta \text{ with } \bar \theta \in \ext(\Im D^* \cap B_{\ell^1}).
	\]
	Thus we have two conditions of extremality, which are of different nature but of course equivalent, as it can be shown directly.
	We have already derived in \cref{rmk:CSextremality} that if
	$(\Ker D^*)^\bot \cap (D \bar s)^\bot \cap \Ker D^*_{\bar J} = \{0\}$, then
	$\bar x \in \ext((\Ker D^*)^\bot \cap B_1)$; it follows that $\bar \theta = D^* \bar x \in \ext(\Im D^* \cap B_{\ell^1})$
	and is such that $(D^*)^+\bar \theta = \bar x$ (recall that $(D^*)^+ D^*$ is the orthogonal projection onto $\Im D = (\Ker D^*)^\bot$).
	Conversely, let $\bar \theta \in \ext(\Im D^* \cap B_{\ell^1})$ and $\bar x = (D^*)^+\bar \theta$
	(note that $D^* \bar x= \bar \theta$ since $D^*(D^*)^+$ is the orthogonal projection onto $\Im D^*$).
	Let $d \in (\Ker D^*)^\bot \cap (D \bar s)^\bot \cap \Ker D^*_{\bar J}$.
	Note that $\cosupp(\bar \theta) = \bar J \subset \cosupp(D^* d)$.
	Then $\sign(\bar \theta + \varepsilon D^* d)= \bar s$ for $|\varepsilon|$ small,
	and by \cref{lem:sign&norm}, $\| \bar \theta + \varepsilon D^* d \|_1 = \langle \bar s,  \bar \theta + \varepsilon D^* d \rangle
	= \langle \bar s,  \bar \theta \rangle+ \varepsilon \langle D \bar s,  d \rangle = \| \bar \theta \|_1$,
	i.e. $\bar \theta + \varepsilon D^* d \in B_{\ell^1}$ for for $|\varepsilon|$ small	(and $\bar \theta + \varepsilon D^* d \in \Im D^*$).
	By extremality of $\bar \theta$, $D^*d = 0$, i.e. $d \in \Ker D^* \cap (\Ker D^*)^\bot = \{0\}$,
	which ends the proof.
\end{remark}

\subsection{Testing the extremality}
\label{sec:testing}

In this subsection, we aim to show that the results of \cref{sec:cons-unit} can be exploited to numerically test the extremality of a point.

Our first definition formalizes the idea of feasible sign, i.e., signs which are attained by some vector in the ambient space.
\begin{definition}
  We say that a sign $s \in \ens{-1,0,1}^p$ is \emph{feasible} with respect to $D$ if there exists $x \in \RR^n$ such that $s = \sign(D^* x)$.
\end{definition}
Note that we can replace at no cost $\RR^n$ by $B_1$ or $\partial B_1$ thanks to the homogeneity of the the $\lun$-norm.
Testing if a sign is feasible has the complexity of a linear program on $n$ variables with $p$ constraints:
\begin{lemma}
  Let $c \in \RR^n$, $s \in \ens{-1,0,1}^p$, $J = \cosupp(x)$, $I^+ = \enscond{i}{\dotp{d_i}{x} > 0}$ and $I^- = \enscond{i}{\dotp{d_i}{x} < 0}$. The sign $s$ is feasible if, and only if, the solution set of
  \begin{equation}\label{eq:linprog-test-sign}
    \umin{x \in \RR^n} \dotp{c}{x} \qsubjq
    \begin{cases}
      D_J^* x &= 0 \\
      D_{I^+}^* x &\geq +1 \\
      D_{I^-}^* x &\leq -1 
    \end{cases}
  \end{equation}
  is non-empty.
  Moreover, if $s$ is feasible, then any solution $x$ of~\eqref{eq:linprog-test-sign} is such that $s = \sign(D^* x)$.
\end{lemma}
We defer to \cref{sec:struct} how the choice of $c$ can leads to interesting properties.
Here, we wrote the problem as a linear program to put an emphasis that existing solvers allow us to test this property.
Note that finding all feasible signs is quite costly since it needs an exponential (in $p$) number of linear programs.

Thanks to \cref{cor:NSCextremality}, we have the definition
\begin{definition}
  We say that a sign $s \in \ens{-1,0,1}^p$ is \emph{pre-extreme} if it satisfies
	\[ (\Ker D^*)^\bot \cap (D \bar s)^\bot \cap \Ker D^*_{\bar J} = \{0\} , \]
	where $J = \cosupp(s)$, and is \emph{extreme} if it feasible and pre-extreme.
\end{definition}

Checking if a sign is pre-extreme boils down to compute the null-space of the matrix
\[ 
	B = 
	\begin{pmatrix}
		U & D \bar{s} & D_{\bar J}
	\end{pmatrix}^*
\]
where $U$ is a basis of the null-space of $D$.
In order to find the dimension of $\Ker B$, one can use either QR reduction or SVD (Singular Value Decomposition).
Here, we used an SVD approach.

We now illustrate these definitions in low dimension (it is known that the study of the number of faces is a very difficult task in general~\cite{mcmullen1970maximum,billera1980}), when $D$ correspond to the incidence matrix of a complete graph on $n=4$ vertices and $p=6$ edges.
Among $3^p = 729$ possible signs, only $75$ are feasible, and among them $14$ are extreme.
We report in \cref{fig:extremal4} the pattern of such signs up to centrosymmetry of the unit-ball, i.e., we only show 7 of the 14 extreme signs.
\begin{figure}[h]
	\centering
	\includegraphics[width=0.9\linewidth]{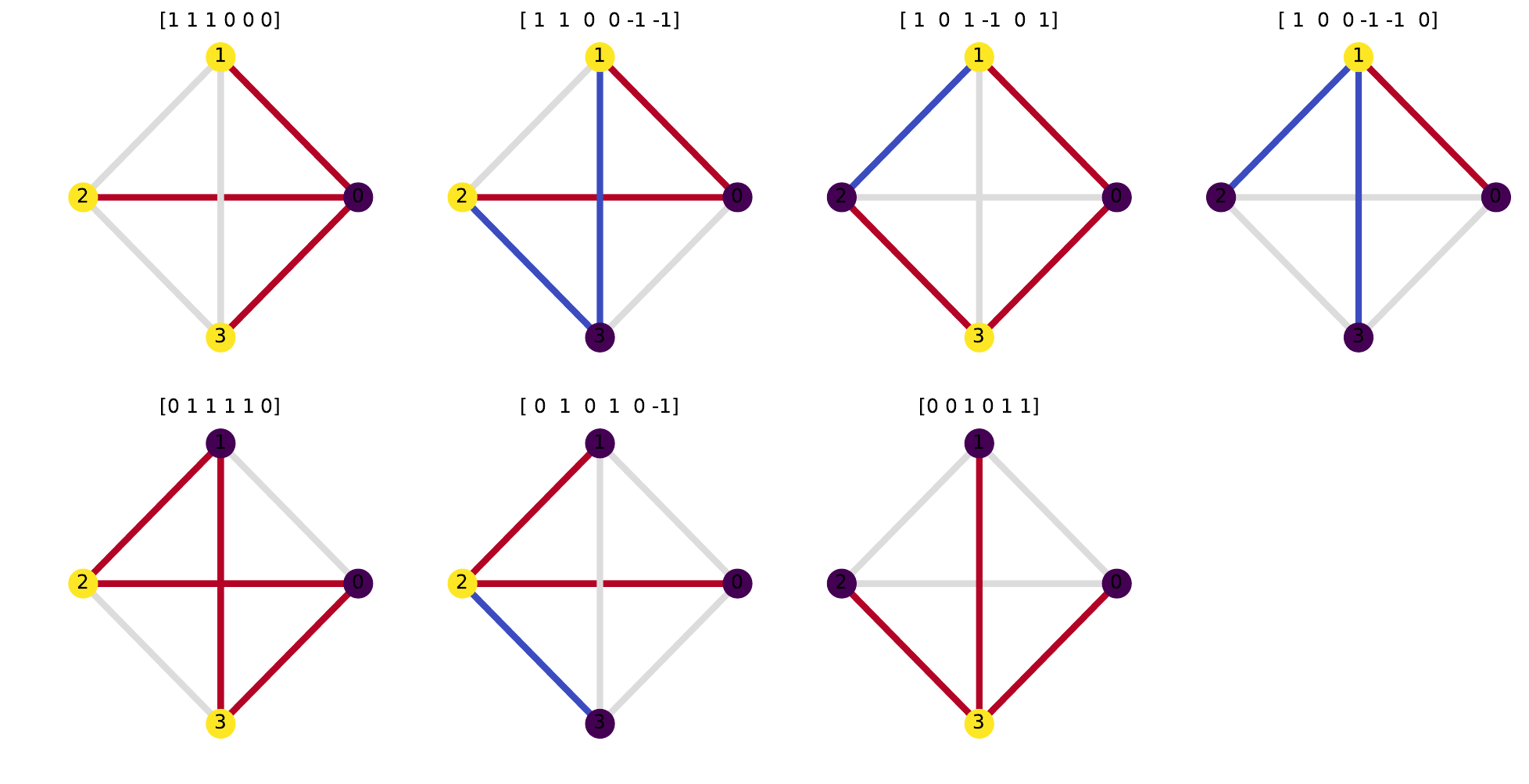}
	\caption{Extreme signs of a complete graph on $n=4$ vertices. Red edges correspond to positive sign, blue edges to negative sign and gray to 0.}
	\label{fig:extremal4}
\end{figure}
\section{The solution set of sparse analysis regularization}
\label{sec:struct}

This section is an application of the previous one towards the solution set of~\eqref{eq:regularization}.
Remark that a similar analysis can be performed (in a less challenging way) for the noiseless problem~\eqref{eq:regularizationNONOISE}.
In the first \cref{sec:inter}, we show that the solution set can be seen as a particular sub-polyhedron of the unit-ball.
Using this result, we derive several structural results in \cref{sec:cons-sol} on the solution set thanks to \cref{sec:unit}.
Finally, we show that arbitrary sub-polyhedra of the unit ball can be seen as solution set of~\eqref{eq:regularization} with a specific choice of parameters.

\subsection{The solution set as a sub-polyhedron of the unit ball}
\label{sec:inter}

This section studies the structure of the solution set of~\eqref{eq:regularization}, that we denote by $\Sol$:
\begin{equation*}
\Sol = \uargmin{x \in \R^n} \Ll(x) \eqdef \frac{1}{2} \| y - \Phi x \|_2^2 + \lambda \| D^* x \|_1 ,
\end{equation*}
where $y \in \R^q$, $\Phi \colon \R^n \rightarrow \R^q$, $D \colon \R^p \rightarrow \R^n$ is linear and $\lambda >0$.

\begin{theorem}
	\label{thm:sol-set=poly}
	The solution set of~\eqref{eq:regularization} is a nonempty convex polyhedron of the form
	\[
	X = \A \cap B_r
	\]
	with $r \ge 0$ and $\A$ an affine subspace such that $\emptyset \ne \A \cap B_r \subset \partial B_r$ and $\dir(\A) = \Ker \Phi$.
	Namely, if $x \in \Sol$, then $r = \| D^* x \|_1$ and $\A = x + \Ker \Phi$.
\end{theorem}

\begin{proof}
It is easy to see that the objective function $x \mapsto \Ll(x)$ is a nonnegative, convex, closed continuous function with full domain (in particular proper).
However, it is not coercive, hence existence of minimizers and compactness of the solution set are not straightforward.
Following~\cite[Chapter 8]{Rockafellar1970Convexanalysis}, the recession cone $R_\Ll$ of $\Ll$ is given by
\begin{equation*}
  R_\Ll \eqdef \enscond{z \in \RR^n}{\Ll_\infty(z) \leq 0} ,
\end{equation*}
where $\Ll_\infty$ is the recession function of $\Ll$ given by
\begin{equation*}
  \Ll_\infty(z) \eqdef \lim_{t \to + \infty} \frac{\Ll(t z)}{t} \in \RR \cup \ens{+\infty} .
\end{equation*}
It is clear that $R_\Ll$ is non-negative, hence the recession cone $R_\Ll$ is given by $R_\Ll = \enscond{z \in \RR^n}{\Ll_\infty(z) = 0}$.
The lineality space $L_\Ll$ is the subspace of $\RR^{n}$ formed by elements $d$ such that $d \in R_\Ll$ and $- d \in R_\Ll$, i.e, $L_\Ll = R_\Ll \cap (-R_\Ll)$.
The following lemma characterizes the structure of $R_\Ll$ and $L_\Ll$.
\begin{lemma}\label{lem:recession}
  The recession cone $R_\Ll$ and the lineality space $L_\Ll$ of $\Ll$ are given by
  \begin{equation*}
    R_\Ll = L_\Ll = \Ker D^* \cap \Ker \Phi .
  \end{equation*}
\end{lemma}
\begin{proof}
  Let $t > 0$ and $z \in \RR^{n}$.
  We have,
  \begin{align*}
    \frac{1}{t} \Ll(tz)
    &= \frac{1}{2t} \| y - \Phi tz \|_2^2 + \lambda \| D^* z \|_1 \\
    &= \frac{1}{2t} \| y \|_2^2 - \dotp{y}{\Phi z} + t \| \Phi z \|_2^2  + \lambda \| D^* z \|_1 .
  \end{align*}
  Hence,
  \begin{equation*}
    \Ll_\infty(z) = - \dotp{y}{\Phi z} + \lambda \| D^* z \|_1 + \iota_{\Ker \Phi}(z) .
  \end{equation*}
  In particular, $\Ll_\infty(z) = 0$ if, and only if, $z \in \Ker D^* \cap \Ker \Phi$.
  Since $\Ker D^* \cap \Ker \Phi$ is a subspace, we have $R_\Ll = L_\Ll$.
\end{proof}

The polyhedral structure of $\Sol$, as we will see, relies on the following lemma.

\begin{lemma} \label{lem:polyhedral}
	Let $C \subset \R^n$ be a nonempty convex set. Then $\mathcal L$ is constant on $C$ if and only if
	$\Phi$ and $\| D^* \cdot \|_1$ are constant on $C$.
\end{lemma}

\begin{proof}
	Assume that $\mathcal L$ is constant on $C$ and let $x_0 \in C$.
	Suppose that there exists $x_1 \in C$ such that $\Phi x_1 \ne \Phi x_0$
	and let $x = \frac{x_0+x_1}{2}$ (note that $x \in C$).
	Then by strict convexity of $u \mapsto \|y- u \|_2^2$,
	\begin{align*}
		\big\| y - \Phi x\big\|_2^2 &= \big\| y - (\frac{1}{2}\Phi x_0 + \frac{1}{2}\Phi x_1)\big\|_2^2 \\
		& <  \frac{1}{2}\big\| y - \Phi x_0\big\|_2^2 + \frac{1}{2}\big\| y - \Phi x_1\big\|_2^2 .
	\end{align*}
	Together with the convexity inequality of the $\ell^1$ norm:
	\[
	\big\| D^*x\big\|_1 \le \frac{1}{2}\big\|D^* x_0\big\|_1 + \frac{1}{2}\big\| D^* x_1\big\|_1,
	\]
	we get that $\mathcal L(x) < \frac{1}{2}\mathcal L(x_0) + \frac{1}{2}\mathcal L(x_1)$,
	which is in contradiction with $\mathcal L$ constant on $C$.
	Then $\Phi$ is constant on $C$ and thus $\| D^* \cdot \|_1$ too.
	The converse is straightforward.
\end{proof}

We add the following lemma, that gives locally the directions where $\| D^* \cdot \|_1$ is constant.

\begin{lemma} \label{lem:constant-norm}
	Let $\A$ be an affine subspace and $\bar x \in \A$. Then $\| D^* \cdot \|_1$ is constant in a neighborhood of $\bar x$ in $\A$ if and only if
\[
\dir(\A) \subset (D \bar s)^\bot \cap \Ker D^*_{\bar J}
\]
where $\bar s = \sign(D^*\bar x)$ and $\bar J = \cosupp(D^*\bar x)$.
\end{lemma}

\begin{proof}
	By definition, $\| D^* \cdot \|_1$ is constant in a neighborhood of $\bar x$ in $\A$ if and only if
	there exists $\varepsilon >0$ such that $B(\bar x,\varepsilon) \cap \A \subset \partial B_r$ with $r = \|D^*\bar x\|_1$.
	We denote by $C = B(\bar x,\varepsilon) \cap \A$ (note that $\bar x \in \ri(C)$).
	By \cref{prop:unique-face_C},
	$C \subset \partial B_r$ if and only if $C \subset \bar F$
	with $\bar F = B_r \cap \{ x \in \R^n : \langle D\bar s , x \rangle = 1 \}$ if $r >0$ and
	$\bar F = \Ker D^*$ if $r =0$.
	Since $\bar x \in \ri(\bar F)$, $C \subset \bar F$ if and only if $\dir(C) \subset \dir(\bar F)$.
	The result follows by noticing that $\dir(C) = \dir(\A)$ (e.g. by \cref{lem:ri-inter}) and $\dir(\bar F) = (D \bar s)^\bot \cap \Ker D^*_{\bar J}$
	(by \cref{lem:convex-subset} in the case $r>0$, obvious in the case $r=0$).
\end{proof}

Together, the previous two lemmas give the following.

\begin{corollary}
	Let $\A$ be an affine subspace and $\bar x \in \A$. Then $\mathcal L$ is constant in a neighborhood of $\bar x$ in $\A$ if and only if
	\[
	\dir(\A) \subset \Ker \Phi \cap (D \bar s)^\bot \cap \Ker D^*_{\bar J}
	\]
	where $\bar s = \sign(D^*\bar x)$ and $\bar J = \cosupp(D^*\bar x)$.
\end{corollary}

We now go back to the proof of \cref{thm:sol-set=poly}.
By \cref{lem:recession}, the recession cone and the lineality space of $\Ll$ coincide.
Then by~\cite[Theorem 27.1(a-b)]{Rockafellar1970Convexanalysis}, the solution set $\Sol$ is nonempty.
Since $\Ll$ is convex (and closed), the solution set is also convex (and closed).
Moreover, $\Ll$ is constant on $\Sol$.
Then by \cref{lem:polyhedral}, $\Phi$ and $\| D^* \cdot \|_1$ are constant on $\Sol$, i.e.
\[
X  \subset  (x + \Ker \Phi) \cap \partial B_r
\]
with $x \in X$ and $r = \|D^*x\|_1$.
But since $\Ll(x)$ is the minimum of $\Ll$,
\[
(x + \Ker \Phi) \cap B_r \subset X.
\]	
It follows that 
\[
X  =  (x + \Ker \Phi) \cap B_r
\]
with $(x + \Ker \Phi) \cap B_r \subset \partial B_r$, as it was to be proved.

\end{proof}

\subsection{Consequence results on the solution set}
\label{sec:cons-sol}

In this section, we apply the results on sub-polyhedra of the unit ball (\cref{sec:subpolyhedra}) to the solution set $X$ of~\eqref{eq:regularization}.

\begin{proposition}\label{prop:cons-sol-faces}
	Let $\bar x \in \ri(X)$, $\bar s = \sign(D^* \bar x)$,
	and $\bar F = B_r \cap \{ x \in \R^n : \langle D\bar s , x \rangle = r \}$ with $r= \| D^* \bar x\|_1$. Then
	\[
	X = (\bar x + \Ker \Phi) \cap \bar F.	
	\]
	It follows that
	\begin{align*}
	X &=  (\bar x + \Ker \Phi) \cap \{ x \in \R^n : \sign(D^*x) \preceq \bar s \},  \\
	\ri(X) &  = (\bar x + \Ker \Phi)  \cap \{ x \in \R^n : \sign(D^*x) = \bar s \}, \\
	\dir(X) &  = \Ker \Phi \cap \Ker D^*_{\bar J} \text{ (where } \bar J = \cosupp(\bar s) \text{)}.
	\end{align*}
	Moreover, the faces of $X$ are exactly the sets of the form \( \{x \in X : J \subset \cosupp(D^*x) \}\)
	with $\bar J \subset J$; their relative interior is given by \(\{x \in X : J = \cosupp(D^*x) \}\)
	and their direction by \(\Ker \Phi \cap \Ker D^*_J\).
\end{proposition}

\begin{proof}
	First note that the results are trivial in the case $r=0$, as $\bar s = 0$ and $\bar F = B_r = \Ker D^*$.
	Therefore we consider the case $r>0$.
	The first statement follows from \cref{thm:sol-set=poly} and \cref{prop:itself},
	and the second one from \cref{lem:affine_compo}~(iii).
	For the last statement, note that $G$ is a face of $X$ if and only if
	$G = (\bar x + \Ker \Phi) \cap F$ with $F$ a face of $B_r$ such that $F \subset \bar F$.
	Indeed, the direct implication holds for $G = \emptyset$ with $F = \emptyset$,
	for $G = X$ with $F = \bar F$, and for $G$ exposed in $\aff(X)$ by \cref{prop:faces_of_inter}.
	Conversely, let $F = \emptyset$ or $B_r \cap \mathcal H$ (with $\mathcal H$ a supporting hyperplane) be a face of $B_r$.
	Then $(\bar x + \Ker \Phi) \cap F = \emptyset$ or $X \cap \mathcal H$ is a face of $X$.
	The conclusion follows from \cref{prop:inclusion-sign}.
\end{proof}

Thanks to \cref{prop:cons-sol-faces}, we can draw several conclusions.
In particular the role of $\bar{s}$ and $\bar{J}$ allows us to derive properties of the solution set.

The sign $\bar s$ is shared by all the interior solutions of~\eqref{eq:regularization},
which are also maximal solutions ($\bar s = \max_{x \in X} \sign(D^* x)$).
Such a solution can be obtained numerically by the algorithm described in \cite{2017arXiv170300192B}.
Future work should include an analysis of the behavior of more common algorithms such as first-order proximal methods.
	
The knowledge of $\bar s$ (or of $\bar J$ which is the minimal cosupport)
gives the dimension of the solution set~\eqref{eq:regularization} as $\dim(X) = \dim( \Ker \Phi \cap \Ker D^*_{\bar J})$
(up to determining the dimension of the null-space of the matrix
$A  = 	\begin{pmatrix}
			\Phi^* & D_{\bar J}
		\end{pmatrix}^*$,
which again can be done by QR reduction or SVD).
For instance, taking the example of~\cite{2017arXiv170300192B} in Section 6.3, 
\begin{equation}\label{eq:setting3d}
	D^* =
	\left(
	  \begin{array}{lll}
		1& 1& 0\cr 1&0&1\cr2&1&1
	  \end{array}
	\right), \quad
	\Phi=
	\left(
	  \begin{array}{lll}
		1& 1& 1\cr 3&1&1\cr\sqrt{2}&0&0
	  \end{array}
	\right), \quad
	y=\left(
	  \begin{array}{l}
		1\cr1\cr0
	  \end{array}
	\right)
	\qandq
	\lambda=\frac{1}{2} ,
\end{equation}
we can prove that $X = \conv{(0\quad \frac{1}{2} \quad 0)^*, (0\quad 0 \quad \frac{1}{2})^*}$.
But running the interior-point method described in~\cite{2017arXiv170300192B} leads to the specific solution (up to numerical error) $\bar{x} = (0\quad \frac{1}{4} \quad \frac{1}{4})^*$.
In this case, the matrix $A$ reduces to $A = \Phi$ since $\bar{J} = \emptyset$.
Thus, $\dim(X) = \dim(\Ker \Phi) = 1$.

In the previous formula, $\dim(X)$ is decreasing w.r.t. $\bar J$;
it somehow quantifies the tautology according to which the sparser the less sparse solution, the fewer solutions.

Together with \cref{lem:linear_sign}, the previous proposition gives
the following half-space representation of the solution set $X$ of~\eqref{eq:regularization}
(one could of course give similar representations of its faces).
\begin{corollary}\label{cor:lin-rep-sol}
	Let $\bar x \in \ri(X)$, $\bar s = \sign(D^* \bar x)$, $\bar J = \cosupp(\bar s)$ and $\bar I = \supp(\bar s)$.
	Then 
	\[
	X = \{ x \in \R^n : \Phi x = \Phi \bar x, D^*_{\bar J} x = 0,  \diag(\bar s_{\bar I})D^*_{\bar I} x \ge 0 \}.
	\]
\end{corollary}
This result can be used numerically.
Indeed, it provides a linear characterization of the solution set up to the knowledge of a maximal solution.
In the same spirit of~\cite{tibshirani2013}, we can derive bounds on the coefficients (both in the signal domain or in the dictionary domain).
For instance, finding the biggest $i$-coefficient boils down to solve the linear program
\begin{equation*}
	\umax{x \in X} \dotp{x}{e_i} ,
\end{equation*}
where $e_i$ the is the $i$th canonical vector.
Thus, we can describe in a similar fashion which component are dispensable following the vocabulary introduced in~\cite{tibshirani2013}.

We can also apply \cref{lem:affine_compo} to an arbitrary solution (not necessarily interior).
This result is useful when obtaining a solution computed from any algorithm without guarantees on its maximality.
\begin{proposition} \label{prop:rank_deficiency}
	Let $x \in X$, $s = \sign(D^* x)$,
	$F = B_r \cap \{ x \in \R^n : \langle D s , x \rangle = r \}$ with $r= \| D^*x\|_1$. Then
	\[
	G = (x + \Ker \Phi) \cap F
	\]
	is the smallest face of $X$ such that $x \in G$ and the unique face of $X$ such that $x \in\ri(X)$.
	It satisfies
	\begin{align*}
		G  &=  (x + \Ker \Phi) \cap \{ x \in \R^n : \sign(D^*x) \preceq s \},  \\
		\ri(G) &  = (x + \Ker \Phi)  \cap \{ x \in \R^n : \sign(D^*x) = s \}, \\
		\dir(G) &  = \Ker \Phi \cap \Ker D^*_{J} \text{ (where } J = \cosupp(s) \text{)}.
	\end{align*}
\end{proposition}
Note in particular that $\dim(\Ker \Phi \cap \Ker D^*_{J})$ is always the dimension of a subset of solutions.
When $D = \Id$, then $\Ker \Phi \cap \Ker D^*_{J} = \Ker \Phi_{I}$ (with $I = [n]\setminus J$ the support of $x$)
the rank deficiency of $\Phi_I$ (the difference between the size of the support $I$ and the rank of $\Phi_I$)
is a lower bound of the dimension of the solution set. See \cref{sec:real_dataset} for an illustration on a real dataset.

We end this section with a characterization of the compactness of $X$ and of its extreme points,
as well as a sufficient condition for uniqueness knowing a solution.

\begin{proposition}\label{prop:conn-litt}
	The solution set $X$ of~\eqref{eq:regularization} admits extreme points if and only if
	it is compact (i.e. is a convex polytope), if and only if \[ \Ker \Phi \cap \Ker D^* = \{0\}.\]
	A solution $x \in X$ is an extreme point (i.e. $x \in \ext(X)$) if and only if, denoting by $J=\cosupp(D^*x)$,
	\[ \Ker \Phi \cap \Ker D_J^* = \{0\}.\]
\end{proposition}

\begin{proof}
	Recall that $X = (\bar x + \Ker \Phi) \cap B_r$ by \cref{thm:sol-set=poly}.
	The result is trivial in the case $r=0$ and is the transcription of \cref{prop:NSCextremality} in the case $r>0$.
\end{proof}

\begin{corollary}
	Let $x \in X$ and $J=\cosupp(D^*x)$. If $x$ is the unique solution of~\eqref{eq:regularization} (i.e. $X = \{x \}$),
	then $\Ker \Phi \cap \Ker D_J^* = \{0\}$.
\end{corollary}

The condition $\Ker \Phi \cap \Ker D^* = \{0\}$ is equivalent to the recession cone $R_\Ll$ being reduced to $\{0\}$,
which is known to be equivalent to the compactness of the solution set $X$ of~\eqref{eq:regularization},
see e.g. \cite[Theorem 27.1(d)]{Rockafellar1970Convexanalysis}.
Note that this is the condition denoted by ($H_0$) and assumed to hold all throughout~\cite{6380620} or~\cite{vaiter2013local}.
This condition can be specified to our examples:
\begin{itemize}
\item When $D$ is the identity $D = \Id$, it is automatically satisfied ;
\item When $D=D_{TV}$, this condition is satisfied as soon as $\Phi$ does not cancel on constant vectors ;
\item When $D$ is the incidence matrix of a graph, observe that this condition reduces to the fact that $\Phi$ should not be constant on the set of constant vectors in each connected component.
\end{itemize}
		
The condition $\Ker \Phi \cap \Ker D_J^* = \{0\}$ is the one denoted by ($H_J$) and required in \cite{vaiter2013local}
at a given solution in order to undertake a sensitivity analysis.
It turns out from the present study that such solutions are precisely the extreme points of the solution set.
In~\cite{vaiter2013local}, an iterative procedure is proposed in section A.3 to construct such an extreme point.
Alternatively, if one has the knowledge that the maximal solution is quite sparse, then an exhaustive test can be performed in a similar fashion than \cref{sec:testing}.

Going back to the setting proposed in~\cref{eq:setting3d}, we observe that thanks to \cref{prop:conn-litt}, we know that $X$ is compact since $\Ker \Phi$ and $\Ker D^*$ intersect trivially, and we can obtain the extreme points by observing that there is three feasible signs $(1,0,1)$, $(0,1,1)$ and $(1,1,1)$. Only the first two leads to a cosupport $J$ such that $\Ker \Phi$ and $\Ker D_J^*$ intersect trivially.
Using \cref{cor:lin-rep-sol}, one can use any linear solver from these two signs to obtain associated two extreme points of the solution set, i.e.,  $\ext(X) = \ens{(0\quad \frac{1}{2} \quad 0)^*, (0\quad 0 \quad \frac{1}{2})^*}$.

\subsection{Discussion of uniqueness conditions} \label{sec:uniqueness}

We can derive from Proposition~\ref{prop:cons-sol-faces} a necessary and sufficient condition for uniqueness.
Indeed \eqref{eq:regularization} admits a unique solution if and only if $\dir X = \{0\}$.
Therefore if one knows the minimal cosupport $\bar J$ (that is the cosupport of a maximal solution $\bar x$),
then \eqref{eq:regularization} admits a unique solution if and only $\Ker \Phi \cap \Ker D^*_{\bar J} = \{0\}$.
Recall that a maximal solution can be obtained numerically by the algorithm described in \cite{2017arXiv170300192B},
so together with QR reduction or SVD, uniqueness can be checked numerically.

In the case of Lasso problem ($D=\Id$), this necessary and sufficient condition becomes 
$\Ker \Phi_{\bar I} = \{0\}$ with $\bar{I}$ the maximal support.
It is slightly weaker than the first sufficient condition for uniqueness
in the seminal paper of Tibshirani \cite[Lemma~2]{tibshirani2013}.
Indeed, the latter is $\Ker \Phi_{\mathcal E} = \{0\}$ with $\mathcal E$ the so-called
\emph{equicorrelation set}, which always contains the maximal support $\bar I$
and coincides with it for almost every $y$ \cite[Lemma~13]{tibshirani2013}.
In this same paper, Tibshirani proves that his first sufficient condition
is satisfied if the columns of $\Phi$ are in \emph{general position}
\cite[Lemma~3]{tibshirani2013}, which is the case with probability one
if their entries are drawn from a continuous distribution \cite[Lemma~4]{tibshirani2013}.
Note that we make no such assumption in our paper.
The fact that non-uniqueness may arise when $\Phi$ has discrete entries
is confirmed by Ewald and Schneider \cite[Theorem~14]{ewald2020distribution}:
there exists $y$ for which the Lasso problem has a non-unique solution
if and only if $\Im(\Phi^*)$ intersects a face of $[-1,1]^n$ of dimension strictly smaller than
the dimension of $\Ker(\Phi)$; it is in particular the case when $\Phi$ has a row
with only $\pm 1$ entries and a non-trivial nullspace.

In the general case, the derivation of uniqueness conditions 
in the paper of Ali and Tibshirani \cite{2018arXiv180507682A} is complicated
by the fact that there is no (unique) equicorrelation set but several \emph{boundary sets} $\mathcal B$
for which $\Ker \Phi \cap \Ker D^*_{[n]\setminus \mathcal{B}} = \{0\}$ needs to be satisfied.
The authors introduced a notion of \emph{$D^*$-general position} which,
together with the condition $\Ker \Phi \cap \Ker D^*= \{0\}$
(equivalent to the compactness of the solution set by Proposition~\ref{prop:cons-sol-faces}),
implies uniqueness for almost every $y$ \cite[Lemma~6]{2018arXiv180507682A}.
These conditions are satisfied with probability one when the entries of $\Phi$
are drawn from a continuous distribution and $n \le q$ or $n>q$ and 
$\dim(\Ker D^*) \le q$ \cite[Lemmas~7,8]{2018arXiv180507682A} but are not assumed in our work.
Finally, the necessary and sufficient condition of Ewald and Schneider
for the uniqueness of the Lasso minimizer above has been generalized by Schneider, Tardivel et al.
to regularizations by polyhedral norms \cite{schneider2020geometry} 
and then by polyhedral gauges \cite{tardivel:hal-03262087}.
The latter framework includes generalized Lasso for which their condition is
that there exists $y$ such that \eqref{eq:regularization} has a non-unique solution
if and only if $\Im(\Phi^*)$ intersects a face of $D[-1,1]^n$
(the image of the hypercube by $D$, which is a polyhedron in $\R^q$)
of dimension strictly smaller than the dimension of $\Ker(\Phi)$.

\subsection{Arbitrary sub-polyhedra of the unit ball as solution sets}

We have the following converse of \cref{thm:sol-set=poly}.

\begin{theorem}\label{thm:arb}
	Let $r \ge 0$ and $\A$ be an affine subspace such that $\emptyset \ne \A \cap B_r \subset \partial B_r$.
	Then there exist $\Phi$, $y$ and $\lambda > 0$ such that the solution set of \eqref{eq:regularization} is $\Sol = \A \cap B_r$
	and $\Ker \Phi = \dir(\A)$.
\end{theorem}

\begin{proof}
	We first consider the case $r > 0$.
	Let $\bar x \in \A \cap \partial B_r$ and $\bar s = \sign(D^*\bar x)$. 
	We define $\Phi$, $y$ and $\lambda$ as follows:
	we consider $a_1, \ldots, a_m$ a basis of $\dir(\A)^\bot$ (with $m$ the dimension of this subspace)
	and set $\Phi = \left( a_1| \cdots |a_m \right)^*$; then $\Im(\Phi^*) = \dir(A)^\bot$ and $\Ker(\Phi) = \dir(\A)$.
	It follows from \cref{lem:dirA}~(ii) that $\Im(\Phi^*) \cap \left(\sum_{s \succeq \bar s} \R_+ Ds \right) \ne \{0\}$.
	Let $\beta \in \R^q$ and $\alpha_s \ge 0$ for all $s \succeq \bar s$ be such that 
	\[
	\Phi^* \beta = \sum_{s \succeq \bar s} \alpha_s Ds \ne 0;
	\]
	we can assume (by normalizing) that $\sum_{s \succeq \bar s} \alpha_s = 1$.
	We define $u = \sum_{s \succeq \bar s} \alpha_s s$, so that $Du= \Phi^*\beta$.
	Note also that $u_{\bar I} = \bar s_{\bar I}$ and $ \| u_{\bar J}\|_\infty \le 1$
	(with $\bar I =\supp(D^* \bar x)$ and $\bar J = \cosupp(D^* \bar x)$).
	We now fix any $\lambda >0$ and set $y = \Phi \bar x + \lambda \beta$.
	We denote as always by $X$ the solution set of \eqref{eq:regularization}.
	By construction, we have
	\[
	\Phi^* (\Phi \bar x - y) + \lambda Du = 0.
	\]
	It implies that $0 \in \partial \mathcal L(\bar x)$ since $\partial\big( \| D^* \cdot \|_1 \big)(\bar x)= D\partial\big( \|  \cdot \|_1 \big)(D^*\bar x)$
	and $u \in \partial\big( \|  \cdot \|_1 \big)(D^*\bar x)$
	(see e.g. \cite{hiriart2013convex} or \cite{vaiter2013local}), and thus $\bar x \in X$.
	It follows from \cref{thm:sol-set=poly} that $X = (\bar x + \Ker \Phi) \cap B_{\| D^*\bar x \|_1}$.
	But $\bar x + \Ker \Phi = \A$ since $\bar x \in \A$ and $\Ker \Phi = \dir(\A)$,
	and $\| D^*\bar x \|_1 = r$ since $\bar x \in \partial B_r$,
	which concludes the proof of the case $r>0$.
	
	We now treat the case $r=0$, for which $B_r = \partial B_r = \Ker D^*$.
	We define $\Phi$ as in the previous case, $\lambda >0$ arbitrarily, and $y = \Phi \bar x$ for some $\bar x \in \A \cap \Ker D^*$.
	Then again, $\Phi^* (\Phi \bar x - y) + \lambda Du = 0$, here with $u=0$ ($\| u\|_\infty \le 1$).
	It follows that $\bar x \in X$, and then that $X=\A \cap B_r$ as before.
\end{proof}
	
If we relax the condition $\Ker \Phi = \dir(\A)$, we can get rid of the assumption $\emptyset \ne \A \cap B_r \subset \partial B_r$
and at the same time choose the exposed face $F$ so that $\A \cap F$ is a solution set (thus of arbitrary dimension).
This is the object of the next proposition.

\begin{proposition} \label{prop:arbitrary_sol}
	Let $r > 0$, $F$ be an exposed face of $B_r$ and $\A$ be an affine subspace intersecting $F$.
	Then there exist $\Phi$, $y$ and $\lambda > 0$ such that the solution set of \eqref{eq:regularization} is $\Sol = \Aa \cap F$
	(and $\Ker \Phi \subset \dir(\A)$).
\end{proposition}

\begin{proof}
	Let $\bar s = \max_{x \in F} \sign(D^* x)$, so that $F = B_r \cap \{ x \in \R^n : \langle D\bar s , x \rangle = r \}$
	by \cref{prop:face2sign}.
	Let $\Phi = \left( D\bar s | a_1| \cdots |a_m \right)^*$ with $a_1, \ldots, a_m$ a basis of $\dir(\A)^\bot$, so that
	$\ker \Phi = (D\bar s)^\bot \cap \dir(\A)$;
	let $\lambda >0$, $y = \Phi x + \lambda e_1$ for some $x \in \A \cap F$, and $X$ be the associated solution set.
	Then $\Phi^* (\Phi x - y) + \lambda Du = 0$ with $u = \bar s$.
	Since $\sign(D^* x)\preceq \bar s$, $u_I =  \sign(D^* x)_I$ and $\|u_J\|_\infty \le 1$ (with $I = \supp(D^*x)$ and $J = \cosupp(D^*x)$).
	Then $x \in X$ and $X = (x + \Ker\Phi) \cap B_r = (x+ \dir(\A) )\cap(x+ (D\bar s)^\bot ) \cap B_r = \A \cap F$.
\end{proof}

We now illustrate \cref{prop:arbitrary_sol} on non-periodic Total Variation on 3 points in order to give an intuition of the geometric construction, see \cref{fig:tv-3d}.
Let $D$ be a forward discrete difference operator on 3 points, i.e.,
\begin{equation*}
	D^* = 
	\begin{pmatrix}
		-1 & 1 & 0 \\
		0 & -1 & 1
	\end{pmatrix} .
\end{equation*}
Consider the facet $F$ (in grey on the figure) determined by the sign $\bar{s} = (-1,1)$ and an (affine) hyperplane $\Aa$ (in red on the figure) with normal vector $(0,1,0)$ and origin $(1,1,1)$.
The intersection (in green on the figure) of $F$ and $\Aa$ is then the segment defined by $x_1^* = (1,1,2)$ and $x_2^* = (2,1,1)$.
The proof of \cref{prop:arbitrary_sol} gives us how to design a setting such that the solution set $X$ is exactly $X = \conv \ens{x_1^*, x_2^*}$:
let $\lambda = 1$, 
\begin{equation*}
	\Phi = 
	\begin{pmatrix}
		1 & -2 & 1 \\
		0 & 1 & 0 \\
	\end{pmatrix}
	\qandq
	y = 
	\begin{pmatrix}
		2 \\
		1
	\end{pmatrix} .
\end{equation*}
Checking first-order condition of this setting is tedious, but doable, and leads to $X = \conv \ens{x_1^*, x_2^*}$.

\begin{figure}[htp]
 \centering
 \tdplotsetmaincoords{60}{100}
 \begin{tikzpicture}[scale=2,tdplot_main_coords]
   \draw[dotted,->] (0,0,0) -- (1,0,0) node[anchor=north east]{$e_1$};
   \draw[dotted,->] (0,0,0) -- (0,1,0) node[anchor=north west]{$e_2$};
   \draw[dotted,->] (0,0,0) -- (0,0,1) node[anchor=south]{$e_3$};


   \coordinate (A1) at (-2,-2,-1);
   \coordinate (A2) at (3,3,4);
   \coordinate (B1) at (-2,-3,-3);
   \coordinate (B2) at (3,2,2);
   \coordinate (C1) at (-2,-2,-3);
   \coordinate (C2) at (3,3,2);
   \coordinate (D1) at (-2,-1,-1); 
   \coordinate (D2) at (3,4,4); 
   
   \draw[] (A1) -- (A2);
   \draw[] (B1) -- (B2);
   \draw[] (C1) -- (C2);
   \draw[] (D1) -- (D2);

   \draw[fill=lightgray,opacity=0.3] (A1) -- (B1) -- (B2) -- (A2) node[opacity=1.0,right] {$F$} -- cycle;

   \draw (1,0,0) -- (0,0,1) -- (-1,0,0) -- (0,0,-1) -- cycle;

   \draw[fill=red,opacity=0.5] (-1.5,1,-0.5)--(-1.5,1,2.5)--(1.5,1,2.5)--(1.5,1,-0.5) node[opacity=1.0,right] {$\Aa$} --cycle;

   \draw[green,thick] (1.5,1,0.75)--(-1.5,1,.75) node[black,right,midway] {$X = \Aa \cap F$};
   \node[below] at (1.5,1,0.75) {$x_1^*$};
   \node[above] at (-1.5,1,0.75) {$x_2^*$};
 \end{tikzpicture}
 \caption{Construction of a solution set for total variation regularization on 3 points and the associated $\Phi$ and $y$.}
 \label{fig:tv-3d}
\end{figure}

\subsection{Illustration of the main results for 1D Total Variation}
\begin{figure}[t]
	\centering
	\includegraphics[width=\textwidth]{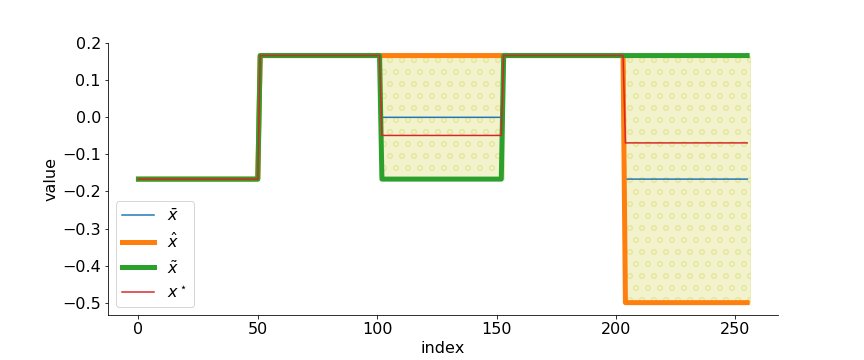}
	\caption{The solution set $X$ represented as the convex hull of $\tilde{x}$ and $\hat{x}$. (Blue) proposed solution $\hat{x}$. (Orange) candidate extreme point $\hat x$. (Green) computed extreme solution. (Red) solution with Chambolle Pock initialized with a random vector. (Yellow) representation of the convex hull.}
	\label{fig:tv1dsol}
\end{figure}
We now provide a full illustration of our results in higher dimension for the popular regularization that is 1D Total Variation.
Note that $D^*$ is a matrix of rank $n-1$ whose nullspace is formed by constant vectors $\Ker D^* = \RR e$ where $e = \begin{pmatrix} 1 & \dots & 1 \end{pmatrix}^*$.

Let $2 \leq t_1 < t_2 < t_3 < t_4 \leq n-1$ and consider the reference signal $\bar{x} \in \RR^n$ and its associated sign
$\bar{s} = \sign( D^* \bar{x}) \in \{-1, 0,+1\}^{n-1}$ defined by
\begin{equation*}
	\bar{x}_v =
	\frac{1}{6} \times 
	\begin{cases}
		-1 & \text{if } 1 \leq v \leq t_1 \\
		\phantom{-}1  & \text{if } t_1 < v \leq t_2 \\
		\phantom{-}0  & \text{if } t_2 < v \leq t_3 \\
		\phantom{-}1  & \text{if } t_3 < v \leq t_4 \\
		-1 & \text{if } t_4 < v \leq n
	\end{cases}
	\qandq
	\bar{s}_e =
	\begin{cases}
		\phantom{-}1 & \text{if } e = t_1 \text{ or } e = t_3 \\
		-1           & \text{if } e = t_2 \text{ or } e = t_4 \\
		\phantom{-}0 & \text{otherwise.}
	\end{cases}	
\end{equation*}
Our objective is to build a problem (i.e, find $\Phi$, $\lambda$ and $y$) such that:
\begin{enumerate}
	\item $\bar{x}$ is a maximal solution (i.e, $\bar{x}$ lives in the relative interior of X).
	\item Every solutions share a common jump at $t_1$.
	\item The affine hull of $X$ is of dimension 1.
\end{enumerate}

The first step is to define a candidate extreme point of $X$ which should be compatible with the sign $\bar{s}$ of $\bar{x}$ according to \cref{cor:sign_map}.
Such vector $\hat{x}$ and sign $\hat{s}$ can be chosen as
\begin{equation*}
	\hat{x}_v =
	\frac{1}{6} \times 
	\begin{cases}
		-1 & \text{if } 1 \leq v \leq t_1 \\
		\phantom{-}1  & \text{if } t_1 < v \leq t_2 \\
		\phantom{-}1  & \text{if } t_2 < v \leq t_3 \\
		\phantom{-}1  & \text{if } t_3 < v \leq t_4 \\
		-3 & \text{if } t_4 < v \leq n.
	\end{cases}
	\qandq
	\hat{s}_e =
	\begin{cases}
		\phantom{-}1 & \text{if } e = t_1 \\
		-1           & \text{if } e = t_4 \\
		\phantom{-}0 & \text{otherwise.}
	\end{cases}	
\end{equation*}
We are now following the proof of \cref{prop:arbitrary_sol} to construct our sparse analysis problem.
We consider the direction $d = z - x$ and build a basis $a_1, \ldots, a_{n-1}$ of $d^\bot$.
This can be done either by hand, or using a SVD decomposition.
We then consider an arbitrary $\lambda > 0$, $\Phi = \left( D\bar s | a_1| \cdots |a_m \right)^*$ and $y = \Phi \bar{x} + \lambda e_1$.
By construction, $\bar{x}$ and $\hat{x}$ are solutions of~\eqref{eq:regularization}, and $\bar{x}$ is a maximal solution.
Since $\Phi$ has rank $n-1$, the affine hull solution set $X$ as at most dimension 1 according to~\cref{prop:cons-sol-faces}.
It is indeed its dimension using a SVD decomposition of $\begin{pmatrix}\Phi \\ D_{\bar J}^*\end{pmatrix}$ where $\bar J = \supp(D^* \bar x)$.
To fully describe it, we have to find its two extreme points.

In order to do it, we are going to use~\cref{cor:lin-rep-sol}.
Let $\bar J = \cosupp(\bar s)$, $\bar I = \bar J^{c}$,
\begin{equation*}
	A_{\text{eq}} =
	\begin{pmatrix}
		\Phi \\
		D^*_{\bar J}
	\end{pmatrix},\quad
	b_{\text{eq}} =
	\begin{pmatrix}
		\Phi \bar{x} \\
		0
	\end{pmatrix}
	\qandq
	A_{\text{ineq}} = - \diag(\bar s_{\bar I})D^*_{\bar I}.
\end{equation*}
Consider now the linear programs for $1 \leq i \leq n-1$:
\begin{equation}\tag{mLP\textsubscript{$i$}}\label{eq:minLP}
	\uargmin{x \in \RR^n} \dotp{x}{d_i} \qsubjq
		A_{\text{eq}} x = b_{\text{eq}} \text{ and }
		A_{\text{ineq}} x \leq 0,
\end{equation}
and
\begin{equation}\tag{MLP\textsubscript{$i$}}\label{eq:maxLP}
	\uargmax{x \in \RR^n} \dotp{x}{d_i} \qsubjq
		A_{\text{eq}} x = b_{\text{eq}} \text{ and }
		A_{\text{ineq}} x \leq 0.
\end{equation}
For a given $1 \leq i \leq n-1$, three cases may occurs:
\begin{enumerate}
	\item[1.] The value of a minimizer of~\eqref{eq:minLP} and a maximizer~\eqref{eq:maxLP} is 0. It means that \emph{every} solution $x \in X$ is zero at the index $i$: for every $x \in X$, $\dotp{x}{d_i} = 0$.
	\item[2a.] The value of a minimizer $m$ of~\eqref{eq:minLP} is negative $m < 0$ and a maximizer~\eqref{eq:maxLP} is 0. It means that \emph{some} solution $x \in X$ is zero at the index $i$, i.e., there exists $x \in X$, $\dotp{x}{d_i} = 0$. We may call the index $i$ a dispensable index to be consistent with the work~\cite{tibshirani2013}.
	\item[2b.] A symmetric situation is when the value of a minimizer of~\eqref{eq:minLP} is 0 and a the value $M$ of a maximizer~\eqref{eq:maxLP} is positive. It means also that \emph{some} solution $x \in X$ is zero at the index $i$.
	\item[3a.] The value of a minimizer $m$ of~\eqref{eq:minLP} is negative $m < 0$ and the value $M$ of a maximizer~\eqref{eq:maxLP} is also negative $M <0$. It means that \emph{no} solution $x \in X$ is sparse at the index $i$, i.e., for every $x \in X$, $m \leq \dotp{x}{d_i} \leq M \neq 0$. The index is called indispensable.
	\item[3b.] The value of a minimizer $m$ of~\eqref{eq:minLP} is positive $m > 0$ and the value $M$ of a maximizer~\eqref{eq:maxLP} is also positive $M > 0$. Then $i$ is indispensable.
\end{enumerate}
Note that according to~\cref{cor:sign_map}, the signs of every solution must be consistent, hence it is impossible to have the situation where the value of a minimizer $m$ of~\eqref{eq:minLP} is strictly negative $m < 0$ and the value $M$ of a maximizer~\eqref{eq:maxLP} is strictly negative $M > 0$.

Beyond the value of~\eqref{eq:minLP} and~\eqref{eq:maxLP}, the actual solution of the linear program is itself a solution of~\eqref{eq:regularization}.
Thus, to find a second candidate to be the extreme point of $X$, it is sufficient to run~\eqref{eq:minLP} and~\eqref{eq:maxLP} for each $1 \leq i \leq n-1$, and consider their nonzero value solutions.
Doing so (using for instance \texttt{scipy.optimize.linprog}~\cite{2020SciPy-NMeth}) let us consider
\begin{equation*}
	\tilde{x}_v =
	\frac{1}{6} \times 
	\begin{cases}
		-1 & \text{if } 1 \leq v \leq t_1 \\
		\phantom{-}1  & \text{if } t_1 < v \leq t_2 \\
		-1  & \text{if } t_2 < v \leq t_3 \\
		\phantom{-}1  & \text{if } t_3 < v \leq t_4 \\
		\phantom{-}1 & \text{if } t_4 < v \leq n
	\end{cases}
	\qandq
	\tilde{s}_e =
	\begin{cases}
		\phantom{-}1 & \text{if } e = t_1 \text{ or } e = t_3 \\
		-1           & \text{if } e = t_2 \text{ or } \\
		\phantom{-}0 & \text{otherwise.}
	\end{cases}	
\end{equation*}
Now, using~\cref{prop:conn-litt}, it is sufficient to check if $\Ker \Phi \cap \Ker D_{J}^*$ intersect trivially for $J = \hat J = \supp(D^* \hat x)$ and $J = \tilde J = \supp(D^* \tilde x)$, which can be done by hand or using again a SVD decomposition of $\begin{pmatrix}\Phi \\ D_J^*\end{pmatrix}$.
This lead to $X = \conv \{ \hat{x}, \tilde{x} \} \ni \bar{x}$.

We illustrate in \cref{fig:tv1dsol} this construction when $n = 256$ and where the solution are obtained by Chambolle-Pock algorithm~\cite{chambolle2011first}, the SVD and linear programming are computed with \texttt{scipy}~\cite{2020SciPy-NMeth}.

\subsection{Illustration on a real dataset} \label{sec:real_dataset}

We now consider the Lasso case where $D = \Id$.
We consider the dataset \texttt{gisette} that is available as a \texttt{libsvm} dataset\footnote{or at the following url: \texttt{\url{https://www.csie.ntu.edu.tw/~cjlin/libsvmtools/datasets/binary/gisette_scale.bz2}}.}.
The \texttt{gisette} dataset was introduced in the NeurIPS 2003 feature selection challenge~\cite{guyon2004result}.
This dataset has $q=5000$ samples and $n=6000$ features.
Its main characteristic is that most entries are $\pm 1$ leading to a rank deficiency.
Figure~\ref{fig:rk_def} shows the rank deficiency of a Lasso problem solved for various proportions $\alpha$ of $\lambda_{\text{max}} = \| \Phi^* y \|_{\infty}$.
This figure is generated with a FISTA solver~\cite{beck2009fast} for $2 \cdot 10^{4}$ iterations in order to have a high accuracy solution (and starting from 0)\footnote{The source code of this experiment is available as a gist:~\texttt{\url{https://gist.github.com/svaiter/e44ee3042a116580aaf33ca48bb4535b}}}.
Such behaviour is not observed for generic datasets (\emph{e.g}. it is not occurring for \texttt{20news} for instance).
\begin{figure}[t]
  \centering
  \includegraphics[width=\textwidth]{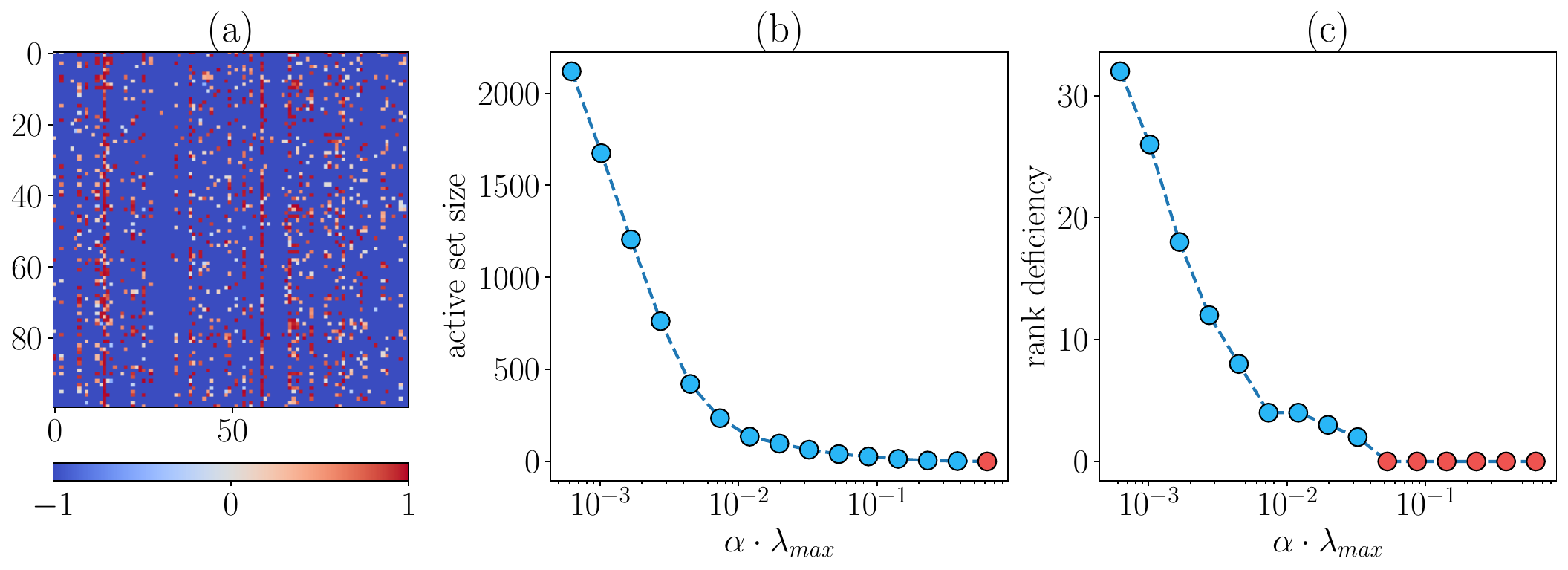}
  \caption{
    Lasso regularization on the \texttt{gisette} dataset for various values of $\lambda$.
    (a) First 100 lines (total 6000) and columns (total 5000) of the matrix $\Phi$ in the dataset. As observed on this sample, around 12\% of the entries are positive.
    (b) Size of the support $I$ with respect to regularization parameter.
    (c) Difference between the size of the support $I$ and the rank of $\Phi_{I}$.
    Red points in (a) and (b) correspond to the value~0.}
  \label{fig:rk_def}
\end{figure}

Solving the Lasso problem for $\lambda = \frac{1}{50} \lambda_{\text{max}}$ leads to a solution $x^{\star}$ having an active set of size 134, but such that $\Phi_{I}$ has rank 130 (where $I = \text{supp}(x^{\star})$).
It is then possible to construct an extreme point using the procedure described in the Appendix A.3 in~\cite{vaiter2013local}.
For the sake of clarity, we recall this ``$H_{J}$-procedure'': if the support $I$ of $x^{\star}$ is such that $\Phi_{I}$ has not full rank,
\begin{enumerate}
\item Take $h \in \Ker \Phi_{I}$ ;
\item Consider the vectors $x_{t} = x^{\star} + t h$, for $t > 0$. There exists a supremum $t_{0}$ such that $x_{t}$ is a solution, and one can prove that $I_{0} = \supp(x_{t_{0}}) \subset I$ .
\end{enumerate}
Iterating this procedure leads to solution with full-rank since the size of the support decreases at least by one at each iteration.


\section{Conclusion}

In this work, we have refined the analysis of the solution set of sparse $\ell^1$ analysis regularization to understand its geometry.
To perform this analysis, we have drawn an explicit relationship between the structure of the unit ball of the regularizer and the set of feasible signs.
Upon this work, we derived a necessary and sufficient condition for a convex set to be the solution of sparse analysis regularization problem.
Extension of our results to non-convex sparse analysis penalizations such as $\| \cdot \|_{p}$ with $0 < p < 1$ is an interesting research direction, where face decomposition of the polytope unit-ball needs to be replaced with stratification of semi-algebraic sets.

From a practical point of view, this work adds another argument towards the need for a good choice of regularizer/dictionary when a user seeks a robust and unique solution to its optimization problem.
This work is mainly of theoretical interest since numerical applications should deal with exponential algorithms with respect to the signal dimension. Note however that in the case of the expected sparsity level of the maximal solution is logarithmic in the dimension, the enumeration problem is in this case tractable.
We believe that the results contained in this paper will help other theoretical works around sparse analysis regularization, such as performing sensitivity analysis of~\eqref{eq:regularization} with respect to the dictionary used in the regularization.

\section*{Acknowledgements}
The authors thank P. Tardivel for pointing out dual geometrical conditions for uniqueness,
and M. Massias for suggesting to use the \texttt{gisette} dataset to illustrate our results.
We also thank the anonymous referees for their valuable comments.

\bibliographystyle{siamplain}
\bibliography{biblio}

\end{document}